\documentclass[10pt,a4paper,reqno]{amsart}
\usepackage{amsmath,amssymb,amsthm,amsfonts}
\usepackage[utf8]{inputenc}
\usepackage{geometry}

\usepackage{enumitem}
\usepackage[english]{babel}
\usepackage{url}
\usepackage{graphicx}
\usepackage{color}

\usepackage{tikz}
\usepackage{pgfplots,caption}
\usepackage{array}
\usepackage{longtable}
\usepackage{ytableau,varwidth}
\usepackage{soul}

\usetikzlibrary{calc}

\newcommand{\Tr}{\operatorname{Tr}}

\newcommand{\Out}{\operatorname{Out}}

\newcommand{\Gal}{\operatorname{Gal}}

\newcommand{\ZZ}{\mathbb{Z}}
\newcommand{\QQ}{\mathbb{Q}}

\newcommand{\V}{\mathrm{V}}
\newcommand{\Li}{\operatorname{Li}}

\makeatletter
\newcommand{\slunlhd}{%
  \mathrel{\mathpalette\sl@unlhd\relax}%
}

\newcommand{\sl@unlhd}[2]{%
  \sbox\z@{$#1\lhd$}%
  \sbox\tw@{$#1\leqslant$}%
  \dimen@=\ht\tw@
  \advance\dimen@-\ht\z@
  \ifx#1\displaystyle
    \advance\dimen@ .2pt
  \else
    \ifx#1\textstyle
      \advance\dimen@ .2pt
    \fi
  \fi
  \ooalign{\raisebox{\dimen@}{$\m@th#1\lhd$}\cr$\m@th#1\leqslant$\cr}%
}
\makeatother

\newtheorem{theorem}{Theorem}[section]
\newtheorem{proposition}[theorem]{Proposition}
\newtheorem{lemma}[theorem]{Lemma}
\newtheorem{corollary}[theorem]{Corollary}

\theoremstyle{definition}

\theoremstyle{remark}
\newtheorem{remark}[theorem]{Remark}

\definecolor{LinkColor}{rgb}{0,0,0} 
\usepackage[colorlinks=true,linkcolor=LinkColor,citecolor=LinkColor,urlcolor= LinkColor, naturalnames, hyperindex, pdfstartview=FitH, bookmarksnumbered, plainpages]{hyperref}

\definecolor{violet}{rgb}{0.56, 0.0, 1.0}
\definecolor{orange(colorwheel)}{rgb}{1.0, 0.5, 0.0}
\definecolor{persianblue}{rgb}{0.11, 0.22, 0.73}

\newcommand{\PSL}{\operatorname{PSL}}

\title{Orders of units in integral group rings and blocks of defect $1$}

\author[M.~Caicedo]{Mauricio Caicedo}
\address{Vakgroep Wiskunde, Vrije Universiteit Brussel, Pleinlaan 2, 1050 Brussels, Belgium.}
\email{\href{mailto:mcaicedo@vub.ac.be}{mcaicedo@vub.ac.be}, \href{mailto:leo.margolis@vub.be}{leo.margolis@vub.be}}

\author[L.~Margolis]{Leo Margolis}
\thanks{Both authors are postdoctoral researchers of the Research Foundation Flanders (FWO - Vlaanderen).}

\begin{document}

\maketitle

\begin{abstract} We show that if a Sylow $p$-subgroup of a finite group $G$ is of order $p$, then the normalized unit group of the integral group ring of $G$ contains a normalized unit of order $pq$ if and only if $G$ contains an element of order $pq$, where $q$ is any prime. We use this result to answer the Prime Graph Question for most sporadic simple groups and some simple groups of Lie type, including seven new infinite series' of such groups. Our methods are based on understanding of blocks of cyclic defect and Young tableaux combinatorics.
\end{abstract}

\section{Introduction}
The problem of describing the unit group of the integral group ring $\mathbb{Z}G$ of a finite group $G$ has led to many interesting results and uncovered many connections between different fields of mathematics, see e.g. the monographs  \cite{Sehgal1993, GRG1, GRG2}. One particular type of questions which attracted a lot of attention is how close the units of finite order in $\mathbb{Z}G$ are to being trivial, i.e. of the form $\pm g$ for some $g \in G$. It has been conjectured by Zassenhaus in 1974 that any unit of finite order is trivial up to conjugation in the bigger algebra $\mathbb{Q}G$, but this turned out recently to be wrong in general  \cite{EiseleMargolis18}. A natural question on the arithmetical properties of torsion units in $\mathbb{Z}G$ is formulated in the so-called Spectrum Problem. To state it, denote by $\varepsilon:\mathbb{Z}G \rightarrow \mathbb{Z}$ the augmentation map which sends an element of a group ring to the sum of its coefficients and by $\V(\ZZ G)$ those units in $\mathbb{Z}G$ that have augmentation $1$. The units in $\V(\ZZ G)$ are also called \textit{normalized} units. Then the Spectrum Problem asks if for any unit of finite order in $\V(\ZZ G)$ there is an element in $G$ of the same order, i.e. if the spectra of $\V(\ZZ G)$ and $G$ coincide. The Spectrum Problem has been positively answered for much bigger classes of groups than the conjecture of Zassenhaus, in particular for solvable groups \cite{HertweckOrders}.

A weaker version of the Spectrum Problem, the so-called Prime Graph Question, has been put forward by Kimmerle \cite{Kimmerle2006}. Recall that the \textit{prime graph} of a, not necessarily finite, group $X$ is the undirected loop-free graph whose vertices are labeled by the primes appearing as order of elements in $X$ and two vertices $p$ and $q$ are connected by an edge if and only if there is an element of order $pq$ in $X$.

\begin{quote}\textbf{Prime Graph Question:} Do $\V(\ZZ G)$ and $G$ have the same prime graph?\end{quote}

It has been observed already in the ground laying work of G. Higman that the vertices of the prime graphs of $\V(\ZZ G)$ and $G$ coincide \cite{Higman1940Thesis} and it was shown later that this is even true for the exponents of $G$ and $\V(\ZZ G)$ \cite{CohnLivingstone}. In other words if $\V(\ZZ G)$ contains an element of order $p^n$, for some prime $p$, then $G$ contains an element of order $p^n$. But the behavior of units of mixed order, i.e. not of prime power order, remains mysterious for non-solvable groups and this paper is a contribution to its understanding. Our main result states that the Prime Graph Question has a positive answer locally around a vertex $p$, if a Sylow $p$-subgroup of $G$ is of order $p$.

\begin{theorem}\label{main_theorem}
Let $G$ be a finite group and let $p$ be a prime. Assume that $G$ has a Sylow $p$-subgroup which is of order $p$. Then, for any prime $q$, there is an element of order $pq$ in $V(\ZZ G)$ if and only if there is an element of order $pq$ in $G$.
\end{theorem}

In contrast to other problems in the field, in particular the Spectrum Problem, a reduction theorem is available for the Prime Graph Question \cite{KimmerleKonovalov}. It states that the Prime Graph Question has a positive answer for a group $G$ if it has a positive answer for all almost simple images of $G$. Recall that a group $G$ is called \textit{almost simple} if it is sandwiched between a non-abelian simple group and its automorphism group, i.e. there is a non-abelian simple group $S$ such that $S \cong \text{Inn}(S) \leq G \leq \text{Aut}(S)$. In this case $S$ is called the \textit{socle} of $G$.

First studies of the Prime Graph Question were based on an algorithmic character-theoretic method. This method was used to answer the Prime Graph Question positively for 13 sporadic simple groups in a series of papers by Bovdi, Konovalov and several coauthors between 2007 and 2012, cf. e.g. \cite{BovdiKonovalovM24, BovdiKonovalovONan}, and their automorphism groups \cite{KimmerleKonovalov15}. By the same method the problem was answered for the simple groups $\PSL(2, p)$ \cite{HertweckBrauer} and any almost simple group having $\PSL(2,p)$ or $\PSL(2,p^2)$ as a socle \cite{BachleMargolis4primaryI}, where $p$ denotes a prime. In combination with the method used in the present paper it was also used to obtain a positive answer for 5 more sporadic simple groups and their automorphism groups \cite{MargolisConway, BachleMargolisSymmetric} and for several almost simple groups with the socle being a simple group of Lie type \cite{KimmerleKonovalov, BachleMargolis4primaryI, BachleMargolis4primaryII}. Recently the problem has been solved for almost simple groups with alternating socle \cite{BachleMargolisSymmetric} and our result is in fact a generalization of the strategy in the last mentioned paper.

Using Theorem~\ref{main_theorem} we obtain an answer for 24 sporadic simple groups and their automorphism groups. Though for many of these groups the Prime Graph Question has been known to hold before, the proofs often relied on calculations which could only be carried out by a computer.

\begin{corollary}\label{sporadic}
Let $G$ be an almost simple group with socle a sporadic simple group $S$. If $S$ is not the O'Nan or the Monster simple group, then the Prime Graph Question holds for $G$. 
\end{corollary} 

If one wants to determine whether Theorem~\ref{main_theorem} is sufficient to answer the Prime Graph Question for almost simple groups of Lie type over a field with $q$ elements, one naturally runs into the question if the product of certain cyclotomic polynomials evaluated at $q$ is a squarefree number, cf. Lemmas~\ref{lemma:PSL4}-\ref{lemma:G2}. The questions if a given integer polynomial has infinitely many squarefree values remains unsolved in full generality. However, in some situation this type of question can be solved. This allows us to answer the Prime Graph Questions for several infinite series' of almost simple groups of Lie-type. The proof of the number-theoretical result on squarefree values of polynomials crucial for our applications was provided to us by Roger Heath-Brown to whom we are very grateful. 


\begin{corollary}\label{Series}
There are infinitely many primes $p$ such that the Prime Graph Question has a positive answer for any almost simple group having one of the following groups as its socle: $\PSL(4,p)$, $\operatorname{PSU}(4,p)$, $\operatorname{PSp}(4,p)$, $\operatorname{PSp}(6,p)$, $\operatorname{P\Omega}(7,p)$, $\operatorname{P\Omega}^+(8,p)$ or $G_2(p)$.
\end{corollary}

A sufficient condition for the primes $p$ in Corollary~\ref{Series} is for $(p^2+1)(p^2-p+1)(p^2+p+1)$ to be a squarefree number. This holds e.g. for $124$ of the 168 primes smaller than 1000.

If one knows the order of an almost simple group $G$ and the orders of the elements therein, then one can try to use Theorem~\ref{main_theorem} to obtain an answer to the Prime Graph Question for $G$. This information is available for many almost simple groups in the GAP Character Table Library \cite{CTblLib}. We summarize the results obtainable for those of these groups which had not been studied before or for which the Prime Graph Question is not known by the results mentioned above.

\begin{corollary}\label{cor:Library}
We list the almost simple groups from the GAP character table library which have a socle not isomorphic to a sporadic group, an alternating group or a group of type $\PSL(2,p)$ or $\PSL(2,p^2)$ for $p$ a prime and for which the Prime Graph Question has not been studied before. If a group appears in the left column this means that the Prime Graph Question is solved for each almost simple group with this socle, if it appears in the right column it means there is an almost simple group with this socle for which the Prime Graph Question remains open. 
\end{corollary}

\begin{center}
\begin{table}[h]
\begin{tabular}{|p{.50\textwidth} | p{.39\textwidth}|}
\hline
{\bf (PQ)} holds by Theorem~\ref{main_theorem} & {\bf (PQ)} not known \\ \hline
$\PSL(3,9)$, $\PSL(4,4)$, $\PSL(4,5)$, $\PSL(4,9)$  & $\PSL(2,125)$, $\PSL(5,3)$ \\
$\PSL(5,2)$, $\PSL(6,2)$, $\PSL(7,2)$, $\PSL(8,2)$ & \\
$\operatorname{PSp}(4,8)$, $\operatorname{PSp}(6,3)$, $\operatorname{PSp}(6,4)$, $\operatorname{PSp}(6,5)$,   & \\
$\operatorname{PSp}(8,2) $, $\operatorname{PSp}(8,3)$, $\operatorname{PSp}(10,2)$, $\operatorname{PSp}(12,2)$ & \\
$\operatorname{PSU}(5,3)$, $\operatorname{PSU}(5,4)$, $\operatorname{PSU}(6,2)$, $\operatorname{PSU}(6,4)$, & $\operatorname{PSU}(3,11)$ \\
$\operatorname{PSU}(7,2)$ & \\
$\operatorname{P\Omega}(7,3)$, $\operatorname{P\Omega}(7,5)$, $\operatorname{P\Omega}^-(8,2)$, $\operatorname{P\Omega}^+(8,3)$,  & \\
$\operatorname{P\Omega}^-(8,3)$, $\operatorname{P\Omega}^+(8,7)$, $\operatorname{P\Omega}(9,3)$, $\operatorname{P\Omega}^+(10,2)$, & \\
$\operatorname{P\Omega}^-(10,2)$, $\operatorname{P\Omega}^-(10,3)$  & \\
$G_2(5)$ & $E_6(2)$, $F_4(2)$, ${}^2G_2(27)$, ${}^2E_6(2)$, ${}^2F_4(8)$ \\
\hline 
\end{tabular}
\caption{Groups from the GAP character table library not studied before. See Table~\ref{table2} for the remaining groups.}
\end{table}
\end{center}

\vspace{-1cm}
For completeness we also include an overview of what can be achieved using Theorem~\ref{main_theorem} for groups from the Character Table Library for which the Prime Graph Question has been studied before, cf. Table~\ref{table2} at the end of the paper.

Theorem~\ref{main_theorem} generalizes \cite[Theorem 1.2]{BachleMargolisSymmetric} which needed an additional assumption on the Brauer tree of the principal $p$-block of $G$ putting this block in a particular Morita equivalence class. 
The methods we use are based on a method introduced in \cite{BachleMargolisLattice}, inspired by an argument in \cite{HertweckA6}, and further developed in \cite{BachleMargolis4primaryII, MargolisConway, BachleMargolisSymmetric}. Roughly speaking this method tries to obtain a contradiction to the existence of a normalized unit $u$ of a certain order in $\ZZ G$  by studying the possible isomorphism types of simple $G$-modules in characteristic $0$ and $p$ when viewed as $\langle u \rangle$-modules. It turns out that this question can be reformulated in terms of Young tableaux combinatorics and the vanishing of certain Littlewood-Richardson coefficients. Moreover a good understanding of the decomposition behavior of simple $G$-modules when passing from characteristic $0$ to characteristic $p$, as it is the case for blocks of cyclic defect, allows to obtain further restrictions. 

In fact before proving Theorem~\ref{main_theorem} we first obtain a quantitative theorem, Theorem~\ref{main_inequality}, on multiplicities of eigenvalues of units of order $pm$ in $p$-blocks of defect $1$, where $m$ is any integer prime to $p$. The specification of this result to the case of $m$ being prime and the principal block allows the proof of Theorem~\ref{main_theorem}. But Theorem~\ref{main_inequality} can also be used for blocks of defect 1 different from the principal block and for orders of units which are not only products of two primes. 

The paper is structured as follows. In Section~\ref{Section_Prelim} we recall the concepts we need in our proofs. This includes knowledge on torsion units in integral group rings, the module structure of group rings of cyclic groups, Littlewood-Richardson coefficients and their connection to modules of cyclic groups and finally the theory of blocks of cyclic defect. In Section~\ref{Section_Prepar} we prove some preparatory results most of which are of a combinatorial nature. We then apply these results in Section~\ref{Section_Proofs} to obtain Theorems~\ref{main_inequality} and \ref{main_theorem}. In Section~\ref{sec_numbertheory} we prove a number-theoretical result on squarefree values of integer polynomials which will be needed for the proof of Corollary~\ref{Series}. Finally Section~\ref{Section_Applications} contains applications of our result to the study of the Prime Graph Question.

\section{Preliminaries and Notation}\label{Section_Prelim}

We introduce the known facts about units in integral group rings, combinatorics and the representation theory of blocks with cyclic defect which will allow us to obtain the proofs of our main results.

Throughout, $G$ is a finite group, if $g \in G$, then $o(g)$ denotes the order of $g$ and $g^G$ denotes the conjugacy class of $g$ in $G$. If $F/K$ is a finite Galois extension of number fields then $\Tr_{F/K}$ denotes the trace map of $F$ over $K$, i.e. $\Tr_{F/K}(x) = \sum_{\sigma \in \text{Gal}(F/K)} \sigma(x)$. For an integer $n$ we denote by $\zeta_n$ an arbitrary but fixed primitive complex $n$-th root of unity.

\subsection{Units in integral group rings}
Let $D: G \rightarrow \operatorname{GL}_n(R)$ be a representation of $G$ over a commutative ring $R$ of characteristic $0$ with character $\chi$. We can linearly extend $D$ to obtain a ring homomorphism $\mathbb{Z}G \rightarrow \operatorname{M}_n(R)$. Being a ring homomorphism it hence restricts to a representation $D: \mathrm{V}(\mathbb{Z}G) \rightarrow \operatorname{GL}_n(R)$ and $\chi$ also linearly extends to a character of $\mathrm{V}(\mathbb{Z}G)$ which we also denote by $\chi$. So if $u \in \mathbb{Z}G$ is unit of order $n$ then $D(u)$ is a matrix of finite order dividing $n$ and has eigenvalues which are $n$-th roots of unity. For an $n$-th root of unity $\zeta$ we denote by $\mu(\zeta, u, \chi)$ the multiplicity of $\zeta$ as an eigenvalue of $D(u)$.

 We will use a formula due to Luthar and Passi which allows one to calculate the multiplicities of eigenvalues of torsion units under a representation. 

 \begin{proposition}[Luthar, Passi {\cite[Theorem 1]{LP89}}]\label{LuPa}  
 Let $G$ be a finite group and $u \in \V(\ZZ G)$ a torsion unit of order $n$. Let $\zeta$ be a complex $n$-th root of unity. 
 Let $\chi$ be an ordinary character and let $D$ be a representation affording $\chi$. Then 
  \begin{equation*} \mu(\zeta, u, \chi) = \frac{1}{n}\sum_{d \mid n} \operatorname{Tr}_{\QQ(\zeta^d)/\QQ}(\chi(u^d)\zeta^{-d}).\end{equation*} 
 \end{proposition}
 
 Quite some information about a torsion unit is encoded in its partial augmentations: For an element $u = \sum_{x \in G} u_x x \in \ZZ G$ we denote by $\varepsilon(u) =  \sum_{x \in G} u_x$ the \textit{augmentation} of $u$ and for a conjugacy class $g^G$ of $G$, 
 $$\varepsilon_{g}(u) = \sum_{x \in g^{G}} u_x$$
  denotes the \emph{partial augmentation} of $u$ at the conjugacy class $g^G$. 
  
Using partial augmentations we then obtain for an ordinary character $\chi$ and $u \in \V(\mathbb{Z}G)$ that
$$\chi(u) = \sum_{g^G} \varepsilon_g(u) \chi(g) $$
where the sum runs over all the conjugacy classes of $G$.
  
 Certain partial augmentations of torsion units in $\ZZ G$ are known to vanish.
\begin{lemma}\label{lemma_partaugs}
Let $u \in \V(\ZZ G)$ be of order $n$. Then
\begin{enumerate}
\item  $\varepsilon_1(u) = 0$ if $u \not= 1$ (Berman-Higman Theorem) \cite[Proposition 1.5.1]{GRG1} and 
\item  $\varepsilon_{g}(u) = 0$, whenever $n$ is not divisible by the order of $g$ \cite[Theorem 2.3]{HertweckBrauer}.
\end{enumerate}
\end{lemma} 
 
Partial augmentations of torsion units are also known to satisfy certain congruences, one of which will be useful to us.
 
 \begin{lemma}\label{congrueneces_pclasses}\cite[Lemma 2.2]{BachleMargolisSymmetric}
  Let $p$ be a prime, $u\in V(\ZZ G)$ a torsion unit of order different from $p$ and let $g_1,\ldots, g_k$ be representatives of the conjugacy classes of elements of order $p$ in $G$. Then 
  $$\sum_{i=1}^{k}\varepsilon_{g_{i}}(u)\equiv 0 \mod p.$$
  Consequently if $h_1,\ldots ,h_\ell$ are representatives of the conjugacy classes of elements of $G$ of order different from $p$, then  
  $$\sum_{i=1}^{\ell}\varepsilon_{h_{i}}(u)\equiv 1 \mod p.$$
 \end{lemma}

\subsection{Modules of cyclic groups} 
We recall some well-known facts about modules of modular group algebras of cyclic groups which might be found in many text books on representation theory and are also included in \cite[Proposition 2.2]{BachleMargolisLattice}.

Our first lemma is a reformulation of the Jordan normal form.

\begin{lemma}\label{kC-modules}
Let $C = \langle c \rangle$ be a cyclic group of order $pm$ where $m$ is an integer coprime to $p$ and let $F$ be a field of characteristic $p$ containing a primitive $m$-th root of unity $\xi$. Then any simple $FC$-module $S$ is $1$-dimensional and determined by the action of $c^p$ on $S$ as $\xi^i$ for a certain $i$. Moreover, any indecomposable $FC$-module $I$ is uniserial and has, up to isomorphism, a unique composition factor. Furthermore, the dimension of $I$ is at most $p$. 
\end{lemma}

\textbf{Notation:} If in the situation of the preceding lemma $I$ is an indecomposable $FC$-module, $d = \text{dim}_F(I)$ and $c^p$ acts on a composition factor of $I$ as $\xi^i$, then we denote $I = I_d^i$. \\ 
 
The next proposition follows from \cite[Propositions 2.3, 2.4]{BachleMargolisLattice} and can also be found in \cite[Proposition 2.2]{MargolisConway}.

\begin{proposition}\label{lattice_method}
 Let $C = \langle c \rangle$ be a cyclic group of order $pm$ such that $p$ does not divide $m$. Let $R$ be a complete local ring of characteristic $0$ containing a primitive $m$-th root of unity $\xi$ such that $p$ is contained in the maximal ideal of $R$ and not ramified in $R$. Adopt the bar-notation for reduction modulo the maximal ideal of $R$.
 
 Let $D$ be a representation of $C$ over $R$ such that the eigenvalues of $D(u)$ in an algebraic closure of the quotient field of $R$, with multiplicities, are $\xi A_1,$ $\xi^2 A_2, \ldots ,\xi^m A_m$ for certain multisets $A_i$ consisting of $p$-th roots of unity. Here  $A_i$ might also be empty. Let $\zeta$ be  a non-trivial $p$-th root of unity. 
  Let $M$ be an $RC$-lattice affording the representation $D$. Then $$M\cong M_1\oplus M_2\oplus \cdots \oplus M_m$$ such that for every $i$ we have: If $A_i$ contains $\zeta$ exactly $r$ times and $1$ exactly $s$ times then $$\overline{M_i} \cong aI_p^i \oplus (r-a)I_{p-1}^i \oplus (s-a)I_1^i$$ for some non-negative integer $a\le \min\{r,s\}$.
\end{proposition}
In the previous proposition, note that since the sum of the eigenvalues of $D(u)$ is an element in $R$ we know for every $i$ that if $A_i$ contains $\zeta$ exactly $r$ times, then $A_i$ contains also $\zeta^2,\ldots ,\zeta^{p-1}$ exactly $r$ times.


\subsection{Combinatorics}\label{Combinatorics}
We will introduce the combinatorial notation and facts which will allow us to control the representation theory of the modules involved in the proof of our main results. All the basic facts on the objects described here can be found in \cite{Fulton1997}.

Let $d$ be an integer and $\lambda = (\lambda_1, \ldots ,\lambda_r)$ a \emph{partition} of $d$, i.e. the $\lambda_i$ are positive integers such that $\lambda_1 \geq \lambda_2 \geq \ldots \geq \lambda_r$ and $\lambda_1 + \ldots + \lambda_r = d$. A partition $\mu = (\mu_1,\ldots ,\mu_s)$ is called a \emph{subpartition} of $\lambda$ if $s \leq r$ and $\mu_i \leq \lambda_i$ for each $1\leq i \leq r$ where we set $\mu_i = 0$ for $i>s$. A \emph{Young diagram} $S$ associated to $\lambda$ is an arrangement of boxes consisting of $r$ rows such that the first row contains $\lambda_1$ boxes, the second row $\lambda_2$ boxes etc. If we remove from $S$ the leftmost  $\mu_1$ boxes in the first row, the leftmost  $\mu_2$ boxes in the second row etc. we obtain the \emph{skew diagram} associated to $\lambda/\mu$.

If each box of a Young diagram or a skew diagram is assigned an entry from an alphabet, which in our case will always be the positive integers, it is called a \emph{Young tableau} or \emph{skew tableau} respectively. Let $T$ be a Young tableau or skew tableau. $T$ is called \emph{semistandard} if entries in the same row, when read from left to right, are not decreasing and entries in each column, when read from top to bottom, are increasing. We further denote by $w(T)$ the word we obtain from the entries in $T$ when reading them from right to left and from top to bottom. If $b$ is a box in $T$ we write $w(b)$ for the word obtained in this manner when reading until $b$, in particular the entry in $b$ is the last letter of $w(b)$. One says that $T$ satisfies the \emph{lattice property} if for any box $b$ in $T$ the word $w(b)$ contains the letter $1$ at least as many times as the letter $2$, the letter $2$ at least as many times as the letter $3$ etc. For example Figure~\ref{ExampleLatticeProperty} contains a Young tableau $T'$ satisfying the lattice property such that $w(T') = 1 \ 2 \ 1 \ 3 \ 2 \ 4 \ 3 \ 5$ and a skew tableau $T$ satisfying the lattice property such that $w(T) = 1 \ 1 \ 2 \ 1 \ 3 \ 2 \ 4$. Young tableaux and skew tableaux we will deal with will always be semistandard and satisfy the lattice property.

\[
\begin{tikzpicture}[inner sep=0in,outer sep=0in]

\node (n) {\begin{varwidth}{5cm}{
 
\ytableausetup{notabloids}
\text{$T'$ \ = \ }
\begin{ytableau}     
 1 & 2 & 1   \\           
 2 & 3   \\                
 3 & 4 \\
 5  \\                    
\end{ytableau}
\hspace{-8cm}
\text{$T$ \ = \ } 
\begin{ytableau}
 \none & 1 & 1   \\
 1 & 2   \\
 2 & 3 \\
 4  \\
\end{ytableau}}\end{varwidth}};

\end{tikzpicture}
\]
\captionof{figure}{Illustration Young tableaux and lattice property.}
\label{ExampleLatticeProperty}

\vspace{0.5cm}

If $T$ is semistandard and satisfies the lattice property such that $w(T)$ contains exactly $\nu_1$ times the letter $1$, $\nu_2$ times the letter $2$ etc. then $\nu = (\nu_1,\ldots ,\nu_t)$ is the \emph{content} of $T$. Note that the content of $T$ is a partition, as $T$ satisfies the lattice property. If $\mu$ is a subpartition of $\lambda$ and $\nu$ some partition then the number of ways the skew diagram $\lambda /\mu$ can be made into a skew tableau with content $\nu$ is known as the \emph{Littlewood-Richardson coefficient} of $\lambda$, $\mu$ and $\nu$ which is denoted by $c_{\mu, \nu}^{\lambda}$. Littlewood-Richardson coefficients have been intensively studied by many authors, but we will only rely on the fact that $c_{\mu, \nu}^{\lambda}$ is symmetric in $\mu$ and $\nu$, i.e. $c_{\mu, \nu}^{\lambda} = c_{\nu, \mu}^{\lambda}$ \cite[Section 5.2, Corollary 2]{Fulton1997}.

\setlength{\parindent}{12pt}

These combinatorial notions come into our picture in the following way. If $C$ is a cyclic group of order $p$ and $F$ is a field of characteristic $p$ then by Lemma~\ref{kC-modules} an $FC$-module $M$ is described up to isomorphisms by the dimensions of its indecomposable direct summands, say $d_1$,\ldots ,$d_r$. So if we arrange these dimensions in a non-increasing way then $(d_1,\ldots ,d_r)$ is a partition of $dim_F(M)$. We call this the \emph{partition associated to} $M$. This description is useful to study submodules and quotients of $M$ via the following fundamental result.

\begin{theorem}\cite[Theorem 9]{BachleMargolis4primaryII}
Let $F$ be a field of characteristic $p$, $C$ a cyclic group of order $p$ and $M$, $U$ and $Q$ be $FC$-modules. Let $\lambda$, $\mu$ and $\nu$ be the partitions associated to $M$, $U$ and $Q$ respectively. Then $M$ has a submodule isomorphic to $U$ such that the quotient by this module is isomorphic to $Q$ if and only if $c^\lambda_{\mu, \nu} \neq 0$. 
\end{theorem}

We will use this theorem in the following sense: When $M$, $U$ and $Q$ are $FC$-modules with associated partitions $\lambda$, $\mu$ and $\nu$ such that $U \leq M$ and $M/U = Q$, then there is a skew tableau of shape $\lambda/\mu$ filled in a way that it is semistandard and satisfies the lattice property and has content $\nu$. Such a skew tableau is what we will mean when speaking of the skew tableau associated to $M/U = Q$.  

The previous theorem in combination with the symmetry of the Littlewood-Richardson coefficients implies that when we are only interested in possible isomorphism types we can ``swap'' around submodules and factor modules of an $FC$-module in the following sense. 

\begin{corollary}
Let $F$ be a field of characteristic $p$, $C$ a cyclic group of order $p$ and $M$, $U$ and $Q$ be $FC$-modules. Then $M$ has a submodule isomorphic to $U$ with quotient isomorphic to $Q$ if and only if it has a submodule isomorphic to $Q$ with quotient isomorphic to $U$.  
\end{corollary}   

Applying this lemma as many times as needed we obtain the following.

\begin{lemma}\label{ModSwap}
Let $F$ be a field of characteristic $p$, $C$ a cyclic group of order $p$ and $M$, $U_1,\ldots ,U_n$, $V$ be $FC$-modules. Let $\sigma$ be a permutation of $\{1,\ldots,n \}$ and set $Q_0 = M$. Then there are $FC$-modules $E_1,\ldots ,E_n$, $Q_1, \ldots ,Q_n$ such that for each $i$ 
\begin{itemize}
\item[(i)] $E_i$ is a submodule of $Q_{i-1}$,
\item[(ii)] $Q_{i} = Q_{i-1}/E_{i}$,
\item[(iii)] $Q_n \cong V$,
\item[(iv)] $E_i \cong U_i$
\end{itemize} 
 if and only if there are modules $E_1, \ldots ,E_n$, $Q_1, \ldots ,Q_n$ satisfying (i), (ii), (iii) and $E_i \cong U_{\sigma(i)}$.
\end{lemma}

Let $T$ be a Young tableau or skew tableau with content $\nu = (\nu_1,\ldots ,\nu_t)$. For an integer $s$ we write $\gamma_s(T)$ for the number of $\nu_i$ which satisfy $\nu_i \geq s$. If $M$ is an $FC_{p}$-module we also write $\gamma_s(M)$ for the number of direct indecomposable summands of $M$ of dimension at least $s$. So $\gamma_s(M)$ equals $\gamma_s(T)$ if $T$ is the Young tableau of the same shape as the Young diagram associated to $M$ which has been filled to be semistandard and satisfy the lattice property.

We moreover use the following notation to describe the relative position of two boxes $b$ and $b'$ in a diagram or tableau. Let us say that $b'$ is \emph{West} of $b$ if the column of $b'$ is strictly to the left of the column of $b$, and we say that $b'$ is \emph{west} of $b$ if the column of $b'$ is left or equal to the column of $b$. We use the corresponding notations for the compass directions, and we combine them, using capital and small letters to denote strict or weak inequalities. For example, we say $b'$ is \emph{northWest} of $b$ if the row of $b'$ is above or equal to the row of $b$, and the column of $b'$ is strictly left of the column of $b$. We apply this notation also with respect to the positions of rows or columns.

\subsection{Cyclic blocks}\label{CyclicBlocks}
We will recall parts of the representation theory of blocks with cyclic defect, ``one of the deepest results in modular representation theory'' \cite[p.243]{Navarro98}, relevant to us. The details of this theory can be found in \cite[Chapter VII]{Feit82} or \cite[Chapter 11]{Linckelmann19}. Our applications will be in fact to blocks of defect $1$ and these are also nicely described in \cite[Section 4.12]{LuxPahlings}. The special situation of principal blocks of defect $1$ is also described in \cite[Chapter 11]{Navarro98}.

Let $p$ be a prime, $G$ a finite group and $(K,R,F)$ a $p$-modular system big enough for $G$ and all its subgroups, i.e. $R$ is a complete discrete valuation ring, $K$ the quotient field of $R$ and $F$ the residue field of $R$ such that $F$ has characteristic $p$ and all ordinary representations of $G$ and it subgroups can be realized over $K$. Let $B$ be a $p$-block of $G$ with cyclic defect group $P$. Then there is a combinatorial structure associated to $B$ called the \emph{Brauer tree} $Y = (V,E)$. It is a tree in the sense of graph theory with $V$ being a set of vertices and $E$ a set of edges and we will speak of the vertices having only one neighbor as \emph{leafs}. In particular $Y$ contains no loops or double edges. The elements of $V$ correspond to all the ordinary complex irreducible characters $G$ lying in $B$. Every element of $V$ corresponds to exactly one irreducible complex character except, possibly, exactly one vertex which is called the \emph{exceptional vertex} and will be denoted by $v_x$. Then $v_x$ corresponds to complex irreducible characters $\eta_1, \ldots ,\eta_t$ which are called \emph{exceptional characters} while the other complex irreducible characters of $B$ are called \emph{non-exceptional}. In particular every ordinary complex character in $B$ belongs to exactly one element of $V$. Furthermore the elements of $E$ correspond bijectively to the irreducible $p$-Brauer characters of $B$ or equivalently to simple $FG$-modules in $B$. If $\chi$ and $\psi$ are ordinary irreducible characters of $B$ corresponding to different vertices $v_\chi$ and $v_\psi$ in $Y$ and $M_\chi$ and $M_\psi$ are $RG$-modules realizing $\chi$ and $\psi$ respectively then $v_\chi$ and $v_\psi$ are connected by an edge labeled by the simple $FG$-module $S$ if and only if $S$ is a composition factor of $F \otimes_R M_\chi$ and $F \otimes_R M_\psi$. 
Moreover the labels of the edges adjacent to $v_\chi$ are exactly the composition factors of $F \otimes_R M_\chi$ each appearing with multiplicity 1. 

Using Brauer characters the following fact easily follows from the structure of $Y$ and is also recorded in \cite[Chapter VII, Theorem 2.15 (iii)]{Feit82}. Assign to each vertex $v$ in $V$ a sign $\delta_v \in \{\pm 1\}$ such that neighboring vertices are associated to different signs and choose for each vertex $v$ an ordinary irreducible character $\chi_v$ associated to it. Then for each $g \in G$ of order not divisible by $p$ we obtain
$$\sum_{v \in V} \delta_v \chi_v(g) = 0.$$

This leads immediately to the observation which will be a key ingredient in the proof of our main result:

\begin{lemma}\label{CharacterFromulaBrauerTree}
Using the notation for a Brauer tree introduced above set $\chi_{v_x} = \eta_1 + \ldots + \eta_t$ and denote by $\chi_v$ otherwise the irreducible complex character associated to a non-exceptional vertex $v$. Then for any element $g \in G$ of order coprime to $p$ we have
$$ \delta_{v_x}\chi_{v_x}(g) + t \cdot \sum_{v \in V \setminus \{v_x\}} \delta_v \chi_v(g) = 0. $$  
\end{lemma}

Note that the signs in a Brauer tree can be distributed in two different ways. Hence for a fixed vertex we can choose a sign and then label the other vertices by signs determined from this starting point.

Other facts about Brauer trees which will be important to us are: The number of edges of the tree equals $|N_G(P)/C_G(P)P|$. So the number of edges is at most $p-1$ and the number of vertices at most $p$. The non-exceptional ordinary characters of $B$ have $p$-rational values. Moreover when $p$ is odd and $P$ is of order $p$, then the exceptional characters $\theta_1, \ldots ,\theta_t$ are Galois-conjugate. Even more, if $r$ is the biggest divisor of $|G|$ not divisible by $p$, then 
$$\theta_1 + \ldots + \theta_t = \sum_{\sigma \in \text{Gal}(\mathbb{Q}(\zeta_{rp})/\mathbb{Q}(\zeta_r))} \theta_1^\sigma.$$ 
This last fact follows from \cite[Theorem 4.12.1]{LuxPahlings}

\section{Preparatory results}\label{Section_Prepar}
The following lemma will allow us to handle the exceptional vertex in the Brauer tree using our methods.

\begin{lemma}\label{lemma_GaloisRep}
Let $p$ be an odd prime and $G$ a finite group of order $p^a m$ where $p$ does not divide $m$. Let $K = \mathbb{Q}(\zeta_m)$ and let $\chi$ be an irreducible complex character of $G$. Denote by $K(\chi)$ the smallest field extension of $K$ which contains all the character values of $\chi$. Then 
$$\psi = \sum_{\sigma \in \Gal(K(\chi)/K)} \chi^\sigma $$
is the character of an irreducible $KG$-representation. 
\end{lemma}
\begin{proof}
By a result of Fong the Schur index of $\chi$ over $F$ equals $1$ \cite[Corollary 10.13]{Isaacs1976} and so there is a simple $K(\chi)$-module $M$ affording $\chi$. By \cite[Exercise 9.6]{Isaacs1976} the module $M$ remains irreducible when viewed as $KG$-module and by \cite[Theorem 9.21]{Isaacs1976} it has character $\psi$. 
\end{proof}

We will use this lemma in the following way.
\begin{corollary}\label{cor_GaloisRep}
Let $p$ be an odd prime, $r$ the maximal divisor of $|G|$ coprime with $p$ and $R$ a complete discrete valuation ring  of characteristic $0$ with maximal ideal containing $p$ such that $R$ contains a primitive $r$-th root of unity and $p$ is unramified in $R$. Assume a Sylow $p$-subgroup of $G$ is of order $p$ and the exceptional characters in the principal $p$-block of $G$ are $\theta_1, \ldots ,\theta_t$. Then $\theta_1+\ldots +\theta_t$ is the character of a simple $RG$-module.
\end{corollary}
\begin{proof}
Let $K$ be the quotient field of $R$. By Lemma~\ref{lemma_GaloisRep} the character $\theta_1+\ldots +\theta_t$ is afforded by a simple $KG$-module. By \cite[Remark 6c)]{BachleMargolis4primaryII} any character of a simple $KG$-module is also the character of a simple $RG$-module.
\end{proof}

We next obtain combinatorial results about skew tableaux which will be relevant for us.
First we provide a generalization of \cite[Lemma 3.3]{MargolisConway}.

\begin{lemma}\label{lemma:SmallBranch1}
Let $T$ be a semistandard skew tableau satisfying the lattice property. Let $b$ be a box in $T$ with entry $e$ such that in the same row as $b$ there are $\ell$ boxes to the right of $b$, where $\ell$ might be $0$. Then $w(b)$ contains $e$ at least $\ell +1$ times.
\end{lemma}

\begin{proof}
For $\ell = 0$ this is clear as $w(b)$ contains $e$ at least once, the $e$ coming from $b$ itself.

So assume $\ell > 0$. Let $e_r$ be the entry in the box $b_r$ which lies right from $b$. Then by induction $e_r$ is contained in $w(b_r)$ at least $\ell$ times. Since $e \leq e_r$, as $T$ is semistandard, and $w(b_r)$ satisfies the lattice property, also $e$ is contained in $w(b_r)$ at least $\ell$ times. Hence $e$ is contained in $w(b)$ at least $\ell + 1$ times.
\end{proof}

\begin{lemma}\label{full_rectangle}
 Let $T$ be a semistandard skew tableau satisfying the lattice property. If $T$ contains a full rectangle of boxes of height $h$ and width $k$, then $\gamma_k(T)\ge h$.
\end{lemma}

\begin{proof}
 Let $b$ be the lowest box in the leftmost  column of the rectangle and let $e$ be its entry. Hence $e\ge h$, since $T$ is semistandard. By Lemma~\ref{lemma:SmallBranch1} and the fact that $w(b)$ satisfies the lattice property,  $w(b)$ contains $e$ and every entry smaller than $e$ at least $k$ times. As $e\ge h$ we get $\gamma_k(T)\ge h$.
\end{proof}

\[
\begin{tikzpicture}
\draw  (0,2) -- (0,0) -- (3,0) -- (3,1) -- (4,1) -- (4,2) -- (5,2) -- (5,4) -- (4,4) -- (4,3) -- (1,3) -- (1,2) -- (0,2
);

\draw[red, dashed] (-0.5,3.5) -- (5.5, 3.5);
\draw[red, dashed] (-0.5,0) -- (5.5, 0);
\draw[red, dashed] (0,-0.5) -- (0, 4.5);
\draw[red, dashed] (3.5,-0.5) -- (3.5, 4.5);

\node[label=south:{\Small{$\ell - k$} }] at (3.5, 5) {};
\node[label=west:{\Small{$c+1$} }] at (-0.5,3.5) {};
\node[label=west:{\Small{$c+h$} }] at (-0.5,0) {};
\node[label=north:{$T$ }] at (2,1.5) {};
 
\end{tikzpicture}
\]
\vspace{-.8cm}

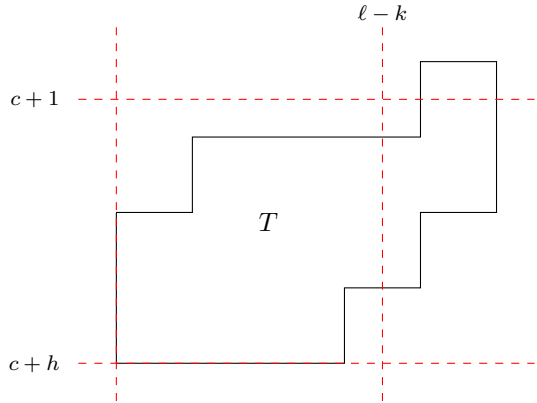
\captionof{figure}{Illustrating Lemma~\ref{colums_between_lines}}
\vspace{.5cm}

\begin{lemma}\label{colums_between_lines}
 Let $T$ be a semistandard skew tableau with $\ell$ columns satisfying the lattice property. Fix some positive integers $c$, $h$ and $k$. Assume that the first $\ell-k$ columns of $T$ lie between the $(c+1)$-th and $(c+h)$-th rows. Then $\gamma_{k+1}(T)\le h$.
\end{lemma}

\begin{proof}
We start by showing the following: Let $b$ be the box in $T$ in which a letter $e$ appears for the $(k+1)$-th time in $w(T)$. Then $b$ is west of the $(\ell-k)$-th column and south of the  $(c+e)$-th row.

\setlength{\parindent}{12pt}

Let $\alpha \leq \ell$ be an integer. Since we are reading from right to left and from top to bottom  and $w(b)$ satisfies the lattice property, if $b$ lies in the $(\ell-\alpha)$-th column, then $w(b)$ contains $e$ at most $\alpha + 1$ times. This follows from the fact that the entries in a box Northwest of $b$ are strictly smaller than $e$, because $T$ is semistandard. Thus, when $e$ appears for the $(k+1)$-th time in $w(T)$, it lies in a box west of the $(\ell-k)$-th column. This proves the first assertion.

To prove the second assertion we argue by induction on $e$. First assume that $e=1$. As west of the $(\ell-k)$-th column there is no box North of the $(c+1)$-th row, the box containing $e$ for the $(k+1)$-th time must be south of the $(c+1)$-th row. Now assume that $e> 1$. By induction hypothesis  $e-1$ appears for the $(k+1)$-th time south of the $(c+e-1)$-th row, say in the box $b_{e-1}$. Then $e$ is contained in $w(b_{e-1})$ at most $k$ times. In addition, every box which is northwest of $b_{e-1}$ contains an entry smaller or equal to $e-1$. Therefore, the box containing $e$ for the $(k+1)$-th time is South of $b_{e-1}$.
 
From the claim stated above it follows that no entry greater than $h$ can appear at least $(k+1)$ times in $w(T)$ is $h$, i.e. $\gamma_{k+1}(T)\le h$.
\end{proof}

\begin{lemma}\label{divided_tableau}
 Let $T$ be a semistandard skew tableau satisfying the lattice property. Divide $T$ into two skew tableaux $T'$ and $T''$ by a vertical line such that $T''$ consists of $k$ columns and $T'$ is the tableau on the right. Then $T'$ is again a semistandard skew tableau with the lattice property and $\gamma_{n+k} (T) \leq \gamma_n(T')$.
\end{lemma}

\begin{proof}
 Clearly $T'$ is semistandard. Moreover, the lattice property is guaranteed by \cite[Section 2.3 Lemma 1, Section 5.2 Lemma 2]{Fulton1997}. Moreover any entry appearing in $T''$ can appear at most $k$ times, as $T''$ is also semistandard. So an entry appearing $n+k$ times in $T$ has to appear at least $n$ times in $T'$ implying $\gamma_{n+k} (T) \leq \gamma_n(T')$. 
\end{proof}

\section{Proofs of main results}\label{Section_Proofs}
In this section we will prove the main results. We first study parts of the Brauer tree and then we show how this leads to the proof of Theorem~\ref{main_inequality} when we consider the whole tree. 

\setlength{\parindent}{12pt}

Throughout the section we will assume that $u \in \mathrm{V}(\mathbb{Z}G)$ is a unit of order $pm$ where $p$ is an odd prime and $m$ an integer not divisible by $p$. We denote by $R$ a complete discrete valuation ring of characteristic $0$ whose maximal ideal contains $p$. Moreover we assume that $R$ contains a primitive $r$-th root of unity where $r$ is the maximal divisor of the order of $G$ not divisible by $p$ and that $p$ is unramified in $R$. We denote by $K$ the quotient field of $R$ and by $F$ the residue field. We use the bar-notation to denote reduction modulo the maximal ideal of $R$, also with respect to modules. Note that if $\chi$ is the character of a simple $KG$-module it is also the character of a simple $RG$-module by \cite[Remark 6c)]{BachleMargolis4primaryII}.  
When we will speak of a Brauer tree, we assume that each vertex $v$ is labeled by a sign $\delta_v$ as described in Section~\ref{CyclicBlocks}.

We will denote by $\xi$ any $m$-th root of unity in $R$, fixed throughout this section. By Proposition~\ref{lattice_method} an $R\langle u \rangle$-module $M$ decomposes as $M \cong M_1 \oplus M_2 \oplus \ldots \oplus M_m$ where $u^p$ is acting on $M_i$ as $\zeta_m^i$. Assume $M_\ell$ is the direct summand of $M$ on which the $p'$-part of $u$ acts as $\xi$. As the action of the $p'$-part of $\langle u \rangle$ on $M_\ell$ is hence fixed we can view $M_\ell$ as an $RC_p$-module via the action of $\langle u^m \rangle$. Thus we can apply also the $\gamma_k$-notation introduced in Section 2 to $M_\ell$ and we will do this using $\gamma_{k,\xi}(M)$, i.e. $\gamma_{k,\xi}(M) = \gamma_k(M_\ell)$.

We start the study of the behavior of multiplicities of eigenvalues of units in the local situation around a fixed vertex of a Brauer tree.

\begin{proposition}\label{negative_node} Let $M$ be an $RG$-module with character $\chi$ such that $\bar{M}$ as $FG$-module has composition factors $E_1$, $E_2, \ldots ,E_n$ and $D$. 
Assume that when viewed as $F \langle u \rangle$-modules $\gamma_{k_i, \xi}(E_i)\le m_i$ for certain $k_i$ and $m_i$ and each $1 \leq i \leq n$, where $k_1 + \ldots + k_i \le p + i-2$ for every $1 \leq i \leq n$.

Then 
$$\gamma_{p-k_1- \cdots - k_n + n-1, \xi}(D) \ge \mu(\xi \cdot \zeta_p, u, \chi) - \sum_{i=1}^{n} m_i.$$ 

The situation is illustrated in the part of a Brauer tree given in Figure~\ref{BrauerTreeProp1}. 
\end{proposition}

\[
\begin{tikzpicture}

\draw (0,1.5)--(3.5,1.5);
\draw (0.5,0.5)--(2,1.5);
\draw (0.5,2.5)--(2,1.5);

\node at (3.5,1.5) (1){};
\node at (2,1.5) (2){};
\node at (0.5,0.5) (3){};
\node at (0,1.5) (4){};
\node at (0.5,2.5) (5){};

\foreach \p in {1,2,3,4,5}{
   \draw[fill=white] (\p) circle (.08cm);
}

\node[label=north:{\Small{$E_1$} }] at (1.4,1.8) {};
\node[label=north:{\Small{$E_i$} }] at (1,1.3) {};
\node[label=north:{\Small{$E_n$} }] at (1.4,0.4) {};
\node[label=north:{\Small{$M$} }] at (2.1,1.4) {};
\node[label=north:{\Small{$D$} }] at (2.7,1.30) {};
\end{tikzpicture}
\]
\vspace*{-.8cm}
\captionof{figure}{Brauer tree illustration for Propositions~\ref{negative_node} and \ref{positive_node}.}\label{BrauerTreeProp1}

\begin{proof}
By Proposition~\ref{lattice_method} we can consider the direct summand $M_\ell$ of $M$, when viewed as $R\langle u \rangle$-module, on which the $p'$-part of $u$ acts as $\xi$. Abusing notation for the rest of the proof we will understand $M$ to be $M_\ell$. We will apply the same for any other module. So when we speak about the composition factors of $\bar{M}$ as $FG$-module viewed as $F\langle u \rangle$-module we will only mean the direct summand of this module on which the $p'$-part of $u$ acts as $\xi$ (or rather the image of $\xi$ in $F$). Then we also have $\gamma_{k,\xi}(M) = \gamma_k(M)$ for any integer $k$.
 
We will use the theory of $FC_p$-modules introduced in Section~\ref{Combinatorics}. 
Assume without loss of generality that $\bar{M}$ has a submodule $E_1$, whose quotient has a submodule $E_2$, whose quotient has submodule $E_n$ etc. and $D$ is in the head of $\bar{M}$, cf. Figure~\ref{SocleTower}.

\[
\begin{tikzpicture}

\draw (0,0)--(0.6,0);
\draw (0,0.6)--(0.6,0.6);
\draw (0,1.2)--(0.6,1.2);
\draw (0,1.8)--(0.6,1.8);
\draw (0,2.4)--(0.6,2.4);
\draw (0,3)--(0.6,3);
\draw (0,0)--(0,3);
\draw (0.6,0)--(0.6,3);
%
\node[label=north:{\Small{$D$} }] at (0.3,2.3) {};
\node[label=north:{\Small{$E_n$} }] at (0.3,1.7) {};
\node[label=north:{\vdots }] at (0.3,1.1) {};
\node[label=north:{\Small{$E_2$} }] at (0.3,0.5) {};
\node[label=north:{\Small{$E_1$} }] at (0.3,-0.1) {};

\end{tikzpicture}
\]
\vspace*{-.8cm}
\captionof{figure}{Assumed composition of module $\bar{M}$} \label{SocleTower}

\vspace{0.3cm}

As we are only interested in the isomorphism types of the $E_i$ as $F\langle u^m \rangle $-modules, i.e. $FC_p$-modules, this is actually possible by Lemma~\ref{ModSwap}. Set $Q_0 = \bar{M}$ and $Q_i=Q_{i-1}/E_i$ for each $1 \leq i \leq n$. So in particular, $Q_n \cong D$. For $i \geq 2$ let $T_i$ denote the skew tableau of $Q_{i-1}/E_i = Q_i$ for $1 \leq i \leq n$, i.e. the skew diagram obtained from removing the Young diagram corresponding to $E_i$ from the Young diagram corresponding to $Q_{i-1}$ and filling it with entries such that it becomes a semistandard skew tableau satisfying the lattice property with entries realizing the isomorphism type of $Q_i$.   

We prove the following by induction on $i$: 
 \begin{equation}\label{local_ineq1}
\gamma_{p-k_1-\cdots - k_i + i-1}(Q_i) \ge \mu(\xi \cdot \zeta_p, u, \chi) - \sum_{j=1}^{i} m_j.
\end{equation}
 
Note that the conclusion of the proposition follows from \eqref{local_ineq1} for $i = n$.
 
We proceed with the base case. By assumption the height of the $k_1$-th column of $E_1$ is at most $m_1$. Moreover as $M$ is an $RG$-module and $R$ satisfies the assumptions of Proposition~\ref{lattice_method} we know that $\bar{M}$ has only indecomposable direct summands of dimensions $1$, $p-1$ and $p$ and $\gamma_2(\bar{M}) = \gamma_3(\bar{M}) = \ldots = \gamma_{p-1}(\bar{M}) = \mu(\xi \cdot \zeta_p, u, \chi)$. So $T_1$ contains a full rectangle of height at least $\mu(\xi \cdot \zeta_p, u, \chi)- m_1$ and width $p- k_1$, cf. Figure~\ref{BaseCaseNegative}.

\vspace*{-.3cm}
 \[
\begin{tikzpicture}
\draw  (0,1.5) -- (0,-1) -- (0.5,-1) -- (0.5, 0.5)-- (4.5,0.5) -- (4.5,2) --(5,2) -- (5,4) -- (4.5,4) -- (4.5, 3.5) -- (3.5, 3.5) -- (3.5, 3) -- (1.5,3) -- (1.5, 2) -- (0.5, 2) -- (0.5, 1.5) --(0,1.5);
%
\draw[red] (2,2.5) -- (4.5, 2.5);
\draw[dashed] (-0.5,2.5) -- (2, 2.5);
\draw[dashed] (-0.5,0.5) -- (2, 0.5);
\draw[dashed] (2,3) -- (2, 4.5);
\draw[red, dashed] (2,3) -- (2, 2.5);
\draw[red] (2,2.5) -- (2, 0.5);
\draw[red] (4.5,2) -- (4.5, 2.5);
\draw[dashed] (4.5,4.5)--(4.5,2.5);
%
\node[label=south:{\Small{$k_1$} }] at (2, 5) {};
\node[label=south:{\Small{$p-1$} }] at (4.5, 5) {};
\node[label=west:{\Small{$m_1+1$} }] at (-0.5,2.5) {};
\node[label=west:{\Small{$\mu(\xi \cdot \zeta_p, u, \chi)$}  }] at (-0.5,0.5) {};
\node[label=south:{\Small{$T_1$} }] at (1, 1.5) {};

\end{tikzpicture}
\]
\captionof{figure}{Illustrating the base case in the proof of Proposition\ref{negative_node}.}\label{BaseCaseNegative}

\vspace{0.3cm}

Thus $\gamma_{p-k_{1}}(Q_1)\ge \mu(\xi \cdot \zeta_p, u, \chi)- m_1$, by Lemma~\ref{full_rectangle}.

\setlength{\parindent}{12pt}
 
So assume $i>1$. Consider the  $k_i$-th column of $E_i$ which has height at most $m_i$, by hypothesis. By induction hypothesis, the height of the  $(p- k_1-\ldots - k_{i-1} + i-2)$-th column  in $T_{i}$ is at least $\mu(\xi \cdot \zeta_p, u, \chi) - \sum_{j=1}^{i-1} m_j$. Hence we have a rectangle of height at least $\mu(\xi \cdot \zeta_p, u, \chi) - \sum_{j=1}^{i} m_j$ and width 
$$p- k_1-\ldots - k_{i-1} + i-2 - (k_i -1) = p- k_1-\ldots - k_{i-1} - k_i  + i-1,$$
cf. Figure~\ref{4.1dib2}.
Note that the last number is a non-negative integer by the assumption that $k_1 + \ldots + k_i \leq p +(i-2)$. Thus, the claim follows from  Lemma~\ref{full_rectangle}.
 
 \[
\begin{tikzpicture}
\draw  (0,2) -- (0,-1) -- (2,-1) -- (2,0) -- (3,0) -- (3,1) -- (4,1) -- (4,2) -- (5,2) -- (5,4) -- (4,4) -- (4,3) -- (1,3) -- (1,2) -- (0,2);

\draw[dashed] (-0.5,2.5) -- (3.5, 2.5);
\draw[dashed] (-0.5,1.5) -- (3.5, 1.5);
\draw[dashed] (1.5,3) -- (1.5, 4.5);
\draw[dashed] (3.5,3) -- (3.5, 4.5);
\draw[red,dashed] (1.5,2.5) -- (1.5, 3);
\draw[red,dashed] (3.5,2.5) -- (3.5, 3);
\node[label=south:{\Small{$p-k_1 - \ldots- k_{i-1} + i-2$} }] at (4.5, 5) {};
\node[label=south:{\Small{$k_i$} }] at (1.5, 5) {};\node[label=west:{\Small{$m_i+1$} }] at (-0.5,2.5) {};
\node[label=west:{\Small{$\mu(\xi \cdot \zeta_p, u, \chi) - \sum_{j=1}^{i-1} m_j$} }] at (-0.5,1.5) {};
\node[label=south:{\Small{$T_i$} }] at (1.5, 1) {};
\draw[red] (1.5,1.5)  -- (3.5, 1.5) -- (3.5, 2.5) -- (1.5,2.5) -- (1.5,1.5);
\end{tikzpicture}
\]
\captionof{figure}{Illustrating the induction step in the proof of Proposition~\ref{negative_node}}\label{4.1dib2}
\end{proof}

\begin{proposition}\label{positive_node}

Let $M$ be an $RG$-module with character $\chi$ such that $\bar{M}$ as $FG$-module has composition factors $E_1$, $E_2, \ldots ,E_n$ and $D$. 
Assume that when viewed as $F \langle u \rangle$-modules $\gamma_{k_i, \xi}(E_i)\ge m_i$ for certain $k_i$ and $m_i$ and each $1 \leq i \leq n$, where $k_1 + \ldots + k_i \ge (i-1)p + 2$ for every $1 \leq i \leq n$.

Then $$\gamma_{np-k_1-\cdots - k_n + 2, \xi}(D) \le \mu(\xi \cdot \zeta_p, u, \chi) - \sum_{i=1}^{n} m_i.$$ 

The situation is illustrated in the part of a Brauer tree given in Figure~\ref{BrauerTreeProp1}. 
\end{proposition}

\begin{proof}
As in the proof of the previous proposition we will work entirely in the part of the modules on which the $p'$-part of $u$ acts as $\xi$. In particular, we have again $\gamma_{k,\xi}(M) = \gamma_k(M)$. Again using Lemma~\ref{ModSwap} we will assume that $\bar{M}$ contains $E_1$ as submodule whose quotient contains $E_2$ whose quotient contains $E_3$ etc. and $D$ is in the head of $\bar{M}$, cf. Figure~\ref{SocleTower}.
Set $Q_0 = \bar{M}$ and $Q_i=Q_{i-1}/E_i$ for $1 \leq i \leq n$. In particular, $Q_n \cong D$. Let $T_i$ denote the corresponding skew tableau of $Q_{i-1}/E_i = Q_i$ for $1 \leq i \leq n$, i.e. as in the proof of the previous proposition $T_i$ is a skew diagram of form corresponding to the Young diagram of $Q_{i-1}$ from which a Young diagram of form corresponding to $E_i$ has been removed and $T_i$ is a semistandard skew tableau satisfying the lattice property realizing the isomorphism type of $Q_i$. We prove the following by induction on $i$:
\begin{equation}\label{local_ineq2}
 \gamma_{ip-k_1-\cdots - k_i + 2}(T_i) \le \mu(\xi \cdot \zeta_p, u, \chi) - \sum_{j=1}^{i} m_j.
\end{equation}

Note that the conclusion of the proposition follows from \eqref{local_ineq2} for $i = n$.

\setlength{\parindent}{12pt}

Consider the  $k_1$-th column of $T_1$. Note that $k_1 \geq 2$ by assumption. By Proposition~\ref{lattice_method} the height of the $k_1$-th column of a Young diagram corresponding to $Q_0$ has height at most $\mu(\xi \cdot \zeta_p, u, \chi)$. So the height of the $k_1$-th column of $T_1$ is at most $\mu(\xi \cdot \zeta_p, u, \chi) - m_1$ as $\gamma_{k_1}(E_1)\ge m_1$ by assumption. Divide $T_1$ into two skew tableaux $T'$ and $T''$ by a vertical line in such a way that $T''$ is the first column of $T_1$ and $T'$ is the rest, cf. Figure~\ref{BaseCasePositive}.

 \[
\begin{tikzpicture}
\draw  (0,1.5) -- (0,-1) -- (0.5,-1) -- (0.5, 0.5)-- (4.5,0.5) -- (4.5,2) --(5,2) -- (5,4) -- (4.5,4) -- (4.5, 3.5) -- (4, 3.5) -- (4, 3) -- (2,3) -- (2, 2) -- (0.5, 2) -- (0.5, 1.5) --(0,1.5);
%

\draw[dashed] (-0.5,3.5) -- (0.5, 3.5);
\draw[red, dashed] (0.5,3.5) -- (3.5, 3.5);
\draw[dashed] (-0.5,0.5) -- (2, 0.5);
\draw[line width=0.1cm, gray] (0.5, -1.5) -- (0.5, 4.5);
\draw[dashed] (3.5,3.5) -- (3.5, 4.5);
\draw[red, dashed] (3.5,3.5) -- (3.5, 0.5);

%

\node[label=south:{\Small{$k_1$} }] at (3.5, 5) {};
\node[label=south:{\Small{$2$} }] at (0.5, 5) {};
\node[label=west:{\Small{$m_1+1$} }] at (-0.5,3.5) {};
\node[label=west:{\Small{$\mu(\xi \cdot \zeta_p, u, \chi)$}  }] at (-0.5,0.5) {};
\node[label=south:{\Small{$T''$} }] at (0.25, 0.3) {};
\node[label=south:{\Small{$T'$} }] at (2.5, 2) {};
\end{tikzpicture}
\]

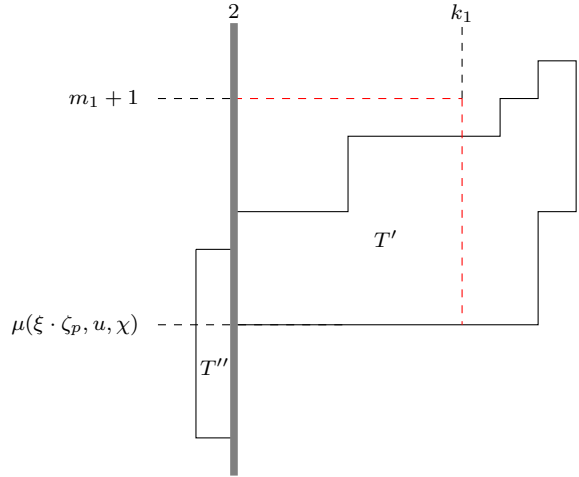
\captionof{figure}{Illustrating the base case in the proof of Proposition~\ref{positive_node}.}\label{BaseCasePositive}

\vspace{0.5cm}

As $T_1$ satisfies the lattice property so does $T'$, by Lemma~\ref{divided_tableau}. Moreover, $T'$ is semistandard, so we are in a situation to apply Lemma~\ref{colums_between_lines} to $T'$. Hence $\gamma_{p-k_{1}+1}(T')\le \mu(\xi \cdot \zeta_p, u, \chi)-m_1$. Further, it follows, by Lemma~\ref{divided_tableau}, that $\gamma_{p-k_{1}+2}(T_1)\le\gamma_{p-k_{1}+1}(T')\le \mu(\xi \cdot \zeta_p, u, \chi)-m_1$ which proves the base case.

\setlength{\parindent}{12pt}

Now assume $i>1$.  Consider the  $k_i$-th column of the skew tableau $T_i$. By induction hypothesis we have that the height of the  $((i-1)p - k_1 -\ldots - k_{i-1} + 2)$-th column of $T_{i}$ is at most $\mu(\xi \cdot \zeta_p, u, \chi) - \sum_{j=1}^{i-1} m_j$. As we are assuming that $k_1 + \ldots + k_i \ge (i-1)p + 2$, the  $((i-1)p - k_1 - \ldots - k_{i-1} + 2)$-th column is west of the $k_i$-th column. Since $\gamma_{k_i}(E_i)\ge m_i$ and the height of the  $((i-1)p - k_1 - \ldots - k_{i-1}+ 2)$-th column in the skew tableau $T_{i}$ is at most $\mu(\xi \cdot \zeta_p, u, \chi) - \sum_{j=1}^{i-1} m_j$, all columns from the $((i-1)p - m_1 - \ldots - m_{i-1}+ 2)$-th column to the  $k_i$-th column are between the $(m_i+1)$-th and $(\mu(\xi \cdot \zeta_p, u, \chi) - \sum_{j=1}^{i-1} m_j)$-th row, cf. Figure~\ref{VirtualTableuaBox}.

Now we divide $T_i$ into two skew tableaux $T'$ and $T''$ by a vertical line in such a way that $T''$ is formed by the first $(i-1)p - k_1 - \ldots - k_{i-1}+ 1$ columns of $T_i$ and $T'$ is the rest. As $T_i$ satisfies the lattice property so does $T'$, by Lemma~\ref{divided_tableau}. Moreover, $T'$ is semistandard, which allows us to apply Lemma~\ref{colums_between_lines} to $T'$. Hence $\gamma_{p-k_{i}+1}(T')\le \mu(\xi \cdot \zeta_p, u, \chi)- \sum_{j=1}^{i} m_j$. Therefore,  
$$\gamma_{(i-1)p - k_1 -\cdots - k_{i-1} + 1 + (p-k_{i}+1)}(T_i) = \gamma_{ip-k_1-\cdots - k_i + 2}(T_i) \le \gamma_{p-k_{i}+1}(T') \leq \mu(\xi \cdot \zeta_p, u, \chi)- \sum_{j=1}^{i} m_j,$$
by Lemma~\ref{divided_tableau}. Thus, the induction step follows.

\[
\begin{tikzpicture}
\draw  (0,2) -- (0,-1) -- (1,-1) -- (1,0) -- (3,0) -- (3,1) -- (4,1) -- (4,2) -- (5,2) -- (5,4) -- (4,4) -- (4,3) -- (1,3) -- (1,2) -- (0,2);

\draw[dashed] (-0.5,3.5) -- (1.5, 3.5);
\draw[red, dashed] (1.5, 3.5)-- (3.5, 3.5);
\draw[dashed] (-0.5,-0.5) -- (1.5,-0.5); 
\draw[red, dashed] (1.5,-0.5) -- (3.5, -0.5);

\draw[dashed]  (1.5,3.5) -- (1.5, 4.5);
\draw[line width=0.1cm, gray] (1.5,4) -- (1.5, -1);
\draw[dashed] (3.5,4.5) -- (3.5, 3.5);
\draw[red,dashed] (3.5,3.5)--(3.5,-0.5);

\node[label=south:{\Small{$k_i$} }] at (3.5, 5) {};
\node[label=south:{\Small{$(i-1)p-k_1 - \ldots - k_{i-1} + 2$} }] at (0.5, 5) {};\node[label=west:{\Small{$m_i+1$} }] at (-0.5,3.5) {};
\node[label=west:{\Small{$\mu(\xi \cdot \zeta_p, u, \chi) - \sum_{j=1}^{i-1} m_j$} }] at (-0.5,-0.5) {};
\node[label=south:{\Small{$T''$} }] at (0.8, 1.3) {};
\node[label=south:{\Small{$T'$} }] at (2.5, 2) {};

\end{tikzpicture}
\]

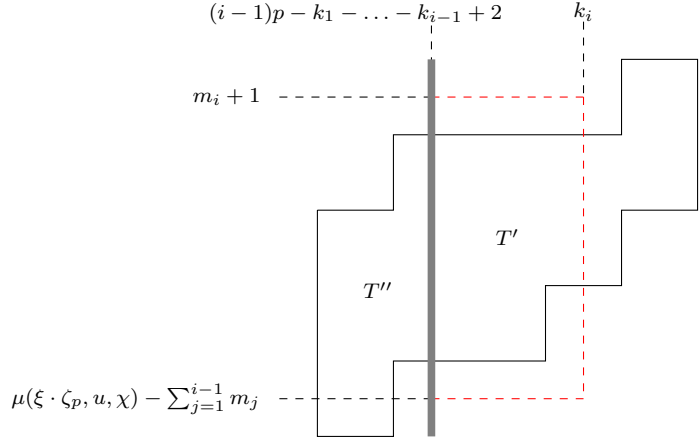
\captionof{figure}{Illustrating proof of induction step in Proposition~\ref{positive_node}}\label{VirtualTableuaBox}
\end{proof}

We now show which consequences can be derived inductively from our previous propositions assuming the exceptional vertex is not involved.

\begin{proposition}\label{general_tree}
Let $M$ be a simple $RG$-module with character $\chi$ lying in a $p$-block $B$ of cyclic defect. Assume that $D$ is a composition factor of $\bar{M}$ when viewed as $FG$-module. We denote the subtree of the Brauer tree of $B$ consisting of the vertices lying to the same side of $D$ as $\chi$ as $S$. I.e. $S$ is the connected component containing $\chi$ of the graph one would get when $D$ would be removed.
We also view $\chi$ as an element of $S$ and moreover assume that $S$ does not contain the exceptional vertex. Let $a$ be the number of vertices in $S$. 

Then:
 \begin{itemize}
 \item[(1.a)]\label{a} If $\delta_{\chi}=-1$, then $\gamma_{p-a, \xi}(D)\ge - \sum_{\psi\in S}\delta_{\psi}\cdot \mu(\xi \cdot \zeta_p, u, \psi)$,
 \item[(1.b)] If $\delta_{\chi}=1$, then $\gamma_{a+1, \xi}(D)\le  \sum_{\psi\in S}\delta_{\psi}\cdot \mu(\xi \cdot \zeta_p, u, \psi)$.
 \end{itemize}

Moreover, if $S$ contains a leaf of the whole Brauer tree of $B$ labeled by a positive sign $+1$, say $\chi_1$, then: 

\begin{itemize}
 \item[(2.a)] If $\delta_{\chi}=-1$, then $\gamma_{p-a+1,\xi}(D)\ge - \mu(\xi , u, \chi_1) - \sum_{\psi\in S}\delta_{\psi}\cdot \mu(\xi \cdot \zeta_p, u, \psi)$,
 \item[(2.b)] If $\delta_{\chi}=1$, then $\gamma_{a,\xi}(D)\le \mu(\xi , u, \chi_1) + \sum_{\psi\in S}\delta_{\psi}\cdot \mu(\xi \cdot \zeta_p, u, \psi)$.
\end{itemize}

\[
\begin{tikzpicture}
\draw (0,0) -- (1.4,0) -- (1.4,1.4) -- (0,1.4) -- (0,0);
\draw (1.4,0.7) -- (3.5, 0.7);


\node at (1.4,0.7) (1){};
\node at (3.5, 0.7) (2){};
\node at (0.7, 0.7) (3){};


\node[label=south:{$\chi$ }] at (1.8, 0.8) {};
\node[label=north:{$D$ }] at (2.5,0.5) {};
\node[label=north:{$S$ }] at (0.7,0.3) {};


\draw[fill=white]  (1) circle (.1cm);
\draw[fill=white]  (2) circle (.1cm);
 
\end{tikzpicture}
\] 
\captionof{figure}{Illustrating Proposition~\ref{general_tree}}\label{BrauerBox}

\end{proposition}

\begin{proof}
Note that as $S$ does not contain the exceptional vertex all the simple $KG$-modules corresponding to characters in $S$ are characters of $RG$-modules.
As in the proof of the preceding proposition we will work only in the parts of $R\langle u \rangle$-modules and $F\langle u \rangle$-modules on which the $p'$-part of $u$ acts as $\xi$. 
Label the edges different from $D$ and adjacent to $\chi$ by $E_1, \ldots ,E_n$ and the subtrees adjacent to these edges by $B_1, \ldots B_n$ respectively, cf. Figure~\ref{BrauerBoxProof}. 

\[
\begin{tikzpicture}

\draw (0,1.5)--(3.5,1.5);
\draw (0.5,0.25)--(2,1.5);
\draw (0.5,2.75)--(2,1.5);

\draw (-0.3, 2.35) -- (0.5, 2.35) -- (0.5, 3.15) -- (-0.3, 3.15) -- (-0.3, 2.35);
\draw (-0.8, 1.1) -- (0,1.1) -- (0,1.9) -- (-0.8, 1.9) -- (-0.8, 1.1);
\draw (-0.3, -0.15) -- (0.5,-0.15) -- (0.5, 0.65) -- (-0.3, 0.65) -- (-0.3, -0.15);

 \node at (3.5,1.5) (1){};
 \node at (2,1.5) (2){};

 \foreach \p in {1,2}{
    \draw[fill=white] (\p) circle (.08cm);
 }

 \node[label=north:{\Small{$E_1$} }] at (1.3,2) {};
 \node[label=north:{\Small{$E_i$} }] at (0.9,1.3) {};
 \node[label=north:{\Small{$E_n$} }] at (1.4,0.3) {};
 \node[label=north:{\Small{$\chi$} }] at (2.1,1.4) {};
 \node[label=north:{\Small{$D$} }] at (2.7,1.35) {};
 \node[label=west:{\Small{$B_1$} }] at (0.6,2.7) {};
 \node[label=west:{\Small{$B_i$} }] at (0.05,1.5) {};
\node[label=west:{\Small{$B_n$} }] at (0.6,0.2) {};
 \end{tikzpicture}
\]
\captionof{figure}{Illustrating proof of Proposition~\ref{general_tree}}\label{BrauerBoxProof}

\vspace{0.5cm}

We first prove (1.a) and (1.b) arguing by simultaneous induction on $a$. So assume first that $a = 1$, i.e. $\chi$ is a leaf and $D \cong \bar{M}$ as $FG$-module. Then by Proposition~\ref{lattice_method} we know $\gamma_2(D) = \gamma_3(D) = ... = \gamma_{p-1}(D) = \mu(\xi \cdot \zeta_p, u, \chi)$. So (1.a) and (1.b) hold. 

 
So assume $a>1$. Note that as the whole Brauer tree of $B$ has at most $p$ vertices we have 
\begin{align}\label{eq_B_i}
|B_1|+|B_2|+ \ldots + |B_n|\le p-2 \ \ \text{and} \ \ |B_1| + \ldots + |B_n| + 1 = a.
\end{align}

Assume first that $\delta_{\chi}=-1$. By induction hypothesis we have by (1.b) that 
$$\gamma_{|B_i|+1}(E_i)\le \sum_{\psi\in B_i}\delta_{\psi}\cdot \mu(\xi \cdot \zeta_p, u, \psi) \ \ \text{for} \; 1\le i \le n.$$

\eqref{eq_B_i} implies that $(|B_1| + 1) + (|B_2| + 1) + \ldots + (|B_i| + 1) \leq p + i - 2$ for any $1 \leq i \leq n$.
Hence by \eqref{eq_B_i} and Proposition~\ref{negative_node}
\begin{multline*}
\gamma_{p-(|B_1|+1)-(|B_2|+1)-\cdots -(|B_n|+1) + n-1}(D)  =   \gamma_{p-a}(D)  \ge \\ 
\mu(\xi \cdot \zeta_p, u, \chi)- \sum_{\psi\in B_1 \cup B_2 \cup \ldots \cup B_n}\delta_{\psi}\cdot \mu(\xi \cdot \zeta_p, u, \psi) =
-\sum_{\psi\in S}\delta_{\psi}\cdot \mu(\xi \cdot \zeta_p, u, \psi),
\end{multline*} which gives us (1.a).

Now assume $\delta_\chi = 1$. By induction hypothesis we know by (1.a) that 
$$\gamma_{p-|B_i|}(E_i)\ge - \sum_{\psi \in B_i}\delta_{\psi}\cdot \mu(\xi \cdot \zeta_p, u, \psi) \ \ \text{for} \; 1\le i \le n.$$
Note that for any $1 \leq i \leq n$ we have $(p-|B_1|)+(p-|B_2|)+ \ldots +(p-|B_i|)\ge (i-1)p + 2$ if and only if $|B_1|+|B_2|+ \ldots +|B_i| + 2 \le p$. As the latter  holds by \eqref{eq_B_i},  we are in a situation to apply Proposition~\ref{positive_node} using the induction hypothesis for (1.a). Hence 
 $$\gamma_{np-(p-|B_1|)-(p-|B_2|)-\cdots -(p-|B_n|) + 2}(D)= \gamma_{a+1}(D	) \le  \sum_{\psi \in S}\delta_{\psi} \cdot \mu(\xi \cdot \zeta_p, u, \psi),$$
 which is the statement of (1.b).

\setlength{\parindent}{12pt}

From now on assume that $S$ contains a leaf labeled by $+1$ called $\chi_1$. We again argue by simultaneous induction on $a$ to prove (2.a) and (2.b). The arguments are similar in this case to the previous ones. If $a = 1$, then necessarily $\delta_\chi = 1$ and as $\gamma_{1}(D) \le \mu(\xi, u, \chi_1) + \mu(\xi \cdot \zeta_p, u, \chi)$ by Proposition~\ref{lattice_method} we know that (2.b) holds.

We first prove (2.a), so assume $\delta_\chi = -1$, $a > 1$ and w.l.o.g. that $\chi_1$ is contained in $B_1$. By (1.b) we have again
$$\gamma_{|B_i|+1}(E_i)\le \sum_{\psi\in B_i}\delta_{\psi}\cdot \mu(\xi \cdot \zeta_p, u, \psi) \ \ \text{for} \; 1\le i \le n$$
and by induction hypothesis on (2.b) we also know
$$\gamma_{|B_1|}(E_1)\le \mu(\xi , u, \chi_1) + \sum_{\psi\in B_1}\delta_{\psi}\cdot \mu(\xi \cdot \zeta_p, u, \psi).$$
Similar as before $|B_1|+(|B_2|+1)+ \ldots + (|B_i|+1)\le p + i - 2$ by \eqref{eq_B_i} for any $1 \leq i \leq n$. So by \eqref{eq_B_i} and Proposition~\ref{negative_node}
$$\gamma_{p-(|B_1|)-(|B_2|+1)-\cdots -(|B_n|+1) + n-1}(D)=\gamma_{p-a+1}(D) \ge - \mu(\xi , u, \chi_1) - \sum_{\psi\in S}\delta_{\psi}\cdot \mu(\xi \cdot \zeta_p, u, \psi),$$ 
as desired.

Next assume $\delta_\chi = 1$. To prove (2.b) we make again use of (1.a) to get
$$\gamma_{p-|B_i|}(E_i)\ge - \sum_{\psi \in B_i}\delta_{\psi}\cdot \mu(\xi \cdot \zeta_p, u, \psi) \ \ \text{for} \; 1\le i \le n$$
and moreover by induction hypothesis on (2.a) we have
 $$\gamma_{p-|B_1|+1}(E_1)\ge - \mu(\xi , u, \chi_1) - \sum_{\psi \in B_1}\delta_{\psi}\cdot \mu(\xi \cdot \zeta_p, u, \psi).$$
Observe that $(p-|B_1|+1)+(p-|B_2|)+ \ldots +(p-|B_i|) \ge (i-1)p + 2$ for any $1 \leq i \leq n$, so we are in a position to apply Proposition~\ref{positive_node} to obtain
$$\gamma_{np-(p-|B_1|+1)-(p-|B_2|)-\cdots -(p-|B_n|) + 2}(D)=\gamma_{a}(D) \le \mu(\xi , u, \chi_1) + \sum_{\psi \in S}\delta_{\psi}\cdot \mu(\xi \cdot \zeta_p, u, \psi).$$
This finishes the proof.

\end{proof}

We are now ready to obtain an equality on the eigenvalues of $u$ for the whole block $B$ in case it is of defect $1$.

\begin{theorem}\label{main_inequality}
Let $G$ be a finite group, $p$ an odd prime and $B$ be a $p$-block of defect 1 with non-exceptional ordinary characters $\chi_1, \ldots ,\chi_e$  such that $\chi_1$ is a leaf. Let $\theta_1, \ldots ,\theta_t$ be the exceptional characters in $B$. Set $\chi_{e+1}=\theta_1+ \ldots +\theta_t$. Let $\delta_i$ be the sign of $\chi_i$ such that $\delta_1=1$. Let $u\in V(\ZZ G)$ be of order $pm$ with $p$ not dividing $m$ and let $\xi$ be any $m$-th root of unity. Then 
$$0 \le \mu(\xi , u, \chi_1) + \delta_{e+1}\cdot \mu(\xi \cdot \zeta_p, u, \chi_{e+1})+ t \sum_{i=1}^{e}\delta_i\cdot \mu(\xi \cdot \zeta_p, u, \chi_i).$$ 
\end{theorem}

\begin{proof}
As in the previous proofs we will work in those summands of $R\langle u \rangle$-modules and $F\langle u \rangle$-modules on which the $p'$-part of $u$ acts as $\xi$.

\setlength{\parindent}{12pt}

By Corollary~\ref{cor_GaloisRep} we know that $\chi_{e+1}$ is the character of an $RG$-module $M$. Label the exceptional vertex of the Brauer tree by $M$. Let $E_1, \ldots ,E_n$ be simple $FG$-modules labeling the edges of the Brauer tree adjacent to the exceptional vertex and let $B_1, \ldots ,B_n$ be subtrees which are adjacent to the edges $E_1, \ldots ,E_n$ respectively, cf. Figure~\ref{AroundExceptional}.

\[
\begin{tikzpicture}

\draw (0,1.5)--(2,1.5);
\draw (0.5,0.25)--(2,1.5);
\draw (0.5,2.75)--(2,1.5);

\draw (-0.3, 2.35) -- (0.5, 2.35) -- (0.5, 3.15) -- (-0.3, 3.15) -- (-0.3, 2.35);
\draw (-0.8, 1.1) -- (0,1.1) -- (0,1.9) -- (-0.8, 1.9) -- (-0.8, 1.1);
\draw (-0.3, -0.15) -- (0.5,-0.15) -- (0.5, 0.65) -- (-0.3, 0.65) -- (-0.3, -0.15);

 \node at (2,1.5) (1){};

\draw (1) circle (.2cm);

 \foreach \p in {1}{
    \draw[fill=white] (\p) circle (.08cm);
 }

 \node[label=north:{\Small{$E_1$} }] at (1.3,2) {};
 \node[label=north:{\Small{$E_i$} }] at (0.9,1.3) {};
 \node[label=north:{\Small{$E_n$} }] at (1.3,0.2) {};
 \node[label=north:{\Small{$M$} }] at (2,1.5) {};
 \node[label=west:{\Small{$B_1$} }] at (0.6,2.7) {};
 \node[label=west:{\Small{$B_i$} }] at (0.05,1.5) {};
\node[label=west:{\Small{$B_n$} }] at (0.6,0.2) {};
\node[label=west:{\Small{$t$} }] at (2.7,1.5) {};
 \end{tikzpicture}
\]

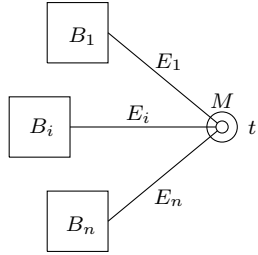
\captionof{figure}{Illustrating proof of Theorem~\ref{main_inequality}}\label{AroundExceptional}

\vspace{0.5cm}

Assume w.l.o.g. that $\chi_1$ lies in $B_1$. The composition factors of $\bar{M}$ as $FG$-module are hence $E_1, \ldots ,E_n$ each with multiplicity $t$. As we are only interested in the isomorphism types of $E_1, \ldots ,E_n$ we can rely on Lemma~\ref{ModSwap} and assume that $E_1$ is in the head of $\bar{M}$. We will study the cases $\delta_{e+1} = -1$ and $\delta_{e+1} = 1$ separately. Our strategy will consist in applying Proposition~\ref{general_tree} where for $E_2, \ldots ,E_n$ we will apply parts (1.a) and (1.b), respectively, each $t$ times, while for $E_1$ we will apply (1.a) exactly $(t-1)$ times, in case $\delta_{e+1} = 1$, or (1.b) $(t-1)$ times, if $\delta_{e+1} = -1$. Comparing with the statement of (2.a) and (2.b) in Proposition~\ref{general_tree} respectively, will then lead to the final inequality.

Assume first $\delta_{e+1}= -1$. 
We apply Proposition~\ref{general_tree} substituting in that Proposition $D$ by $E_i$ and $S$ by $B_i$ to get

\begin{equation}\label{sign_negative1}
 \gamma_{|B_i|+1}(E_i)\le  \sum_{\psi\in B_i}\delta_{\psi}\cdot \mu(\xi \cdot \zeta_p, u, \psi)
 \end{equation}

and 

\begin{equation}\label{sign_negative2}
\gamma_{|B_1|}(E_1)\le  \mu(\xi , u, \chi_1) + \sum_{\psi\in B_1}\delta_{\psi}\cdot \mu(\xi \cdot \zeta_p, u, \psi).
\end{equation}

As $B$ has defect $1$ we have $t(|B_1|+ \ldots +|B_n|)=p-1$. Let $\mathcal{C}$ be a multiset containing every element of $\{2,\ldots ,n \}$ at most $t$ times and the entry $1$ at most $(t-1)$ times. Then
\begin{equation}\label{EstimateBlockCase1}
\sum_{c \in \mathcal{C}} (|B_c|+1) \leq p + |\mathcal{C}| - 2,
\end{equation}
where we use the fact that $|B_1| \geq 1$.

Let
\begin{align*}
&d = t((|B_2|+1)+ \ldots +(|B_n|+1))+ (t-1)(|B_1|+1) =  tn+p-2-|B_1|\le p + (tn-1) -2. \\
\end{align*}

\setlength{\parindent}{12pt}

We intend to use Proposition~\ref{negative_node} to obtain an estimate for $\gamma_{|B_1|}(E_1)$, different from the one in \eqref{sign_negative2}, using each of the inequalities \eqref{sign_negative1} $t$ times for $2 \leq i \leq n$ and $(t-1)$ times for $i = 1$. I.e. we replace in Proposition~\ref{negative_node} the module $D$ by $E_1$, the $k_i$ by $|B_i|+1$ and the $m_i$ by $\sum_{\psi\in B_i}\delta_{\psi}\cdot \mu(\xi \cdot \zeta_p, u, \psi)$, where $i = 2,...,n$ appears $t$ times while $i=1$ appears $t-1$ times. So overall we factor out $tn-1$ modules from $\bar{M}$ before we obtain $E_1$. So by \eqref{sign_negative1}  and using Proposition~\ref{negative_node}, which we can apply by \eqref{EstimateBlockCase1}, we have 

\begin{multline}\label{inq_E1}
\gamma_{|B_1|}(E_1) = \gamma_{p-tn-p+2+|B_1|+ tn-2}(E_1) = \gamma_{p- d + (tn-1)-1}(E_1) \ge \\
\mu(\xi \cdot \zeta_p, u, \chi_{e+1})-(t-1) \sum_{\psi\in B_1}\delta_{\psi}\cdot \mu(\xi \cdot \zeta_p, u, \psi) - t \sum_{\psi\in B_2 \cup \cdots \cup B_n}\delta_{\psi}\cdot \mu(\xi \cdot \zeta_p, u, \psi). 
\end{multline}

Combining \eqref{sign_negative2} and \eqref{inq_E1} we get 

\begin{align*}
& \mu(\xi \cdot \zeta_p, u, \chi_{e+1})-(t-1) \sum_{\psi\in B_1}\delta_{\psi}\cdot \mu(\xi \cdot \zeta_p, u, \psi) - t \sum_{\psi\in B_2 \cup \cdots \cup B_n}\delta_{\psi}\cdot \mu(\xi \cdot \zeta_p, u, \psi)  \le \\ 
& \mu(\xi , u, \chi_1) + \sum_{\psi\in B_1}\delta_{\psi}\cdot \mu(\xi \cdot \zeta_p, u, \psi).\\
\end{align*}

Therefore, 
$$0\le \mu(\xi , u, \chi_1) - \mu(\xi \cdot \zeta_p, u, \chi_{e+1}) + t \sum_{\psi\in B_1 \cup \cdots \cup B_n}\delta_{\psi}\cdot \mu(\xi \cdot \zeta_p, u, \psi).$$
 This finishes the case $\delta_{e+1} = -1$.

Assume $\delta_{e+1}=1$. By Proposition~\ref{general_tree}, replacing $D$ by $E_i$ and $S$ by $B_i$, we have estimates for dimensions of indecomposable direct summands for $E_1, \ldots ,E_n$ given by
\begin{equation}\label{sign_positive1}
 \gamma_{p-|B_i|}(E_i)\ge - \sum_{\psi\in B_i}\delta_{\psi}\cdot \mu(\xi \cdot \zeta_p, u, \psi)
\end{equation}
  and
\begin{equation}\label{sign_positive2}
\gamma_{p-|B_1|+1}(E_1)\ge - \mu(\xi , u, \chi_1) - \sum_{\psi\in B_1}\delta_{\psi}\cdot \mu(\xi \cdot \zeta_p, u, \psi). 
\end{equation}

 Similarly to the previous case we will apply Proposition~\ref{positive_node} to obtain another estimate on $\gamma_{p-|B_1|+1}(E_1)$. To do so, we need to check the assumptions of this proposition. Using again $t(|B_1|+\ldots +|B_n|)=p-1$ for a multiset $\mathcal{C}$ as above we get
\begin{equation}\label{EstimateBlockCase2}
 \sum_{c \in \mathcal{C}} (p-|B_c|) \geq (|\mathcal{C}|-1)p + 2, 
\end{equation} 
 where we again use $|B_1| \geq 1$.

Let
\begin{align*}
&d =t((p-|B_2|)+ \ldots +(p-|B_n|))+ (t-1)(p-|B_1|)=(tn-2)p +|B_1| +1. \\
\end{align*}
 
We apply Proposition~\ref{positive_node} using \eqref{sign_positive1} exactly $t$ times for $2 \leq i \leq n$ and $(t-1)$ times for $i = 1$. I.e. in that proposition $D$ is replaced by $E_1$, the $k_i$ by $p-|B_i|$ and the $m_i$ by $- \sum_{\psi\in B_i}\delta_{\psi}\cdot \mu(\xi \cdot \zeta_p, u, \psi)$, where $i=2,...,n$ appears $t$ times and $i=1$ appears $t-1$ times. The application of Proposition~\ref{positive_node} is possible by \eqref{EstimateBlockCase2}. We obtain
\begin{align*}
&\gamma_{p-|B_1|+1}(E_1) = \gamma_{(tn-1)p- (tn -2)p - |B_1| - 1 + 2}(E_1) = \gamma_{(tn-1)p - d + 2}(E_1) \\
 &\le \mu(\xi \cdot \zeta_p, u, \chi_{e+1})- (t-1) \left(-\sum_{\psi\in B_1}\delta_{\psi} \mu(\xi \cdot \zeta_p, u, \psi)\right) - t \left(- \sum_{\psi\in B_2 \cup \cdots \cup B_n}\delta_{\psi} \mu(\xi \cdot \zeta_p, u, \psi)\right).
\end{align*} 
 
 Combining this with \eqref{sign_positive2} we have 
 
 \begin{align*}
 &-\mu(\xi , u, \chi_1) - \sum_{\psi\in B_1}\delta_{\psi}\cdot \mu(\xi \cdot \zeta_p, u, \psi)\le \\
 &\mu(\xi \cdot \zeta_p, u, \chi_{e+1})+(t-1) \sum_{\psi\in B_1}\delta_{\psi}\cdot \mu(\xi \cdot \zeta_p, u, \psi) + t \sum_{\psi\in B_2 \cup \cdots \cup B_n}\delta_{\psi}\cdot \mu(\xi \cdot \zeta_p, u, \psi).\\
 \end{align*}
 
 Therefore, 
 $$0\le \mu(\xi , u, \chi_1) + \mu(\xi \cdot \zeta_p, u, \chi_{e+1}) + t \sum_{\psi\in B_1 \cup \cdots \cup B_n}\delta_{\psi}\cdot \mu(\xi \cdot \zeta_p, u, \psi).$$
\end{proof}

In the case of the principal block we can obtain a more explicit result, though we have to additionally assume that our units are of order $pq$ for $q$ a prime.
This proof follows the lines of the proof of the main theorem in \cite{BachleMargolisSymmetric}, but uses the stronger preparatory results obtained above. 

\begin{proof} [Proof of Theorem~\ref{main_theorem}:]
First of all, if $p = q$ the result follows from the fact that the exponents of $G$ and $\V(\ZZ G)$ coincide. If $p=2$, then $G$ is solvable and hence the Prime Graph Question has a positive answer by \cite{Kimmerle2006}. If $ p = 3$, then the result is proved in \cite[Theorem D]{BachleMargolis4primaryII}. So from now on we assume $p > 3$. Assume moreover that there is $u\in V(\ZZ G)$ of order $pq$ but there is no element of order $pq$ in $G$.  
 Denote by $P$ a Sylow $p$-subgroup of $G$. As $P$ is cyclic of order $p$, the principal $p$-block $B_0$ has defect 1. Assume that $\mathbf{1}=\chi_1$, $\chi_2, \ldots ,\chi_e$ are the non-exceptional ordinary irreducible characters in $B_0$, where $\mathbf{1}$ is the principal character, and $\theta_1, \ldots ,\theta_t$ are the exceptional irreducible characters in $B_0$. Set $\chi_{e+1}=\theta_1 + \ldots + \theta_t$. 
  
\setlength{\parindent}{12pt}

Let $y$ be a generator of $P$. Note that $\chi_i(y)=\chi_i(y^j)$ for all $1 \le i\le e+1$ and all $1\le j\le p-1$, as  all $\chi_i$ are $p$-rational by the theory of Brauer trees as explained in Section~\ref{CyclicBlocks}. 
 
Define for $a \in \ZZ G$ 
 $$\nu(a)=\delta_{e+1} \chi_{e+1}(a) + t\sum_{i=1}^{e} \delta_i \cdot \chi_i(a),$$
 where the $\chi_i$ are extended linearly to $\ZZ G$.
 
Note that $\chi_{e+1}$ is an irreducible character over $\QQ(\zeta_{|G|_{p'}})$ by Lemma~\ref{lemma_GaloisRep}, where $|G|_{p'}$ denotes the biggest divisor of $|G|$ coprime to $p$. Then $\chi_1, \ldots ,\chi_{e+1}$ are linearly independent over $\QQ(\zeta_{|G|_{p'}})$ by \cite[Corollary 9.22]{Isaacs1976}. So, $\nu(g)=0$ for every $g\in G$ would contradict the linear independence of the characters over $\QQ(\zeta_{|G|_{p'}})$ and this contradiction, which will follow from the existence of $u$, will provide the proof of the theorem. 
Recall that a torsion unit in $V(\ZZ G)$ is called $p$-regular if it has order not divisible by $p$ and $p$-singular otherwise. 

We will show that $\nu(g) = 0$ for every $g \in G$. We divide this is several steps, namely (a) we prove that $\nu(a) = 0$ for every $p$-regular element $a$ of $\V(\ZZ G)$. This, in particular, shows that $\nu(g) = 0$ for $g$ a $p$-regular element of $G$ and will be used to prove in (b) that $\nu(g) = 0$ if $g$ has order $p$. This will then be used to show in (c) that $\nu(g) = 0$ for any $p$-singular element $g \in G$.\\
 
(a) Let $a\in \V(\ZZ G)$ be a $p$-regular torsion unit. By Lemma~\ref{CharacterFromulaBrauerTree}, if $a$ is a group element then $\nu(a)=0$, so assume $a \notin G$. Then, by Lemma~\ref{lemma_partaugs},  $\varepsilon_h(a)=0$ for every $p$-singular element $h$ of $G$. Hence 
 $$\chi(a)=\sum_{x^G, p\nmid o(x)} \varepsilon_x(a)\chi(x) $$
 for every ordinary character $\chi$ of $G$, where the sum runs over all conjugacy classes of $G$ containing elements whose order is not divisible by $p$. Thus,
 \begin{align*}
  \nu(a) &= \delta_{e+1} \left(\sum_{x^G, p\nmid o(x)} \varepsilon_x(a)\chi_{e+1}(x)\right)   + t\sum_{i=1}^{e} \delta_i \left(\sum_{x^G, p\nmid o(x)} \varepsilon_x(a)\chi_i(x) \right)
  \\&=   \sum_{x^G, p\nmid o(x)} \varepsilon_x(a) \left(\delta_{e+1}\cdot \chi_{e+1}(x) + t\sum_{i=1}^{e} \delta_i\cdot \chi_i(x) \right)  
  \\&= \sum_{x^G, p\nmid o(x)} \varepsilon_x(a) \nu(x) =0
 \end{align*}
 
 \noindent
 showing that $\nu$ vanishes on $p$-regular elements of $\V(\ZZ G)$.\\

(b) We show next that $\nu(y)$ which is sufficient to show that $\nu(g) = 0$ for $g \in G$ any element of order $p$. Since every $\chi_i$, for $1\le i\le e+1$, is $p$-rational we have that $\Tr_{\QQ(\zeta_p)/\QQ}(\chi_i(u^q)\zeta_p)=-\chi_i(u^q)$. Further, Theorem~\ref{main_inequality}   for $\xi = 1$ gives 
 \begin{equation}\label{eq_FirstIneq}
 1 = \mu(1 , u, \chi_1) \geq -\delta_{e+1}\cdot \mu(\zeta_p, u, \chi_{e+1}) - t \sum_{i=1}^{e}\delta_i \cdot \mu(\zeta_p, u, \chi_i).
 \end{equation}
 
We now compute the multiplicities involved in \eqref{eq_FirstIneq} using Proposition~\ref{LuPa}. They are: 
 \begin{align*}
  & \mu(\zeta_p, u, \chi_{e+1}) = \\
  & \frac{1}{pq}\left( \chi_{e+1}(u^{pq}) +\Tr_{\QQ(\zeta_q)/\QQ}(\chi_{e+1}(u^{p})) + \Tr_{\QQ(\zeta_p)/\QQ}(\chi_{e+1}(u^{q})\zeta_p^{-q}) + \Tr_{\QQ(\zeta_{pq})/\QQ}(\chi_{e+1}(u)\zeta_p^{-1})\right) = \\
  & \frac{1}{pq}\left(\chi_{e+1}(1) +\Tr_{\QQ(\zeta_q)/\QQ}(\chi_{e+1}(u^{p})) - \chi_{e+1}(u^{q}) + \Tr_{\QQ(\zeta_{pq})/\QQ}(\chi_{e+1}(u)\zeta_p^{-1})\right)\\
 \end{align*}
 
 and 
 
\begin{align*}
&\sum_{i=1}^{e}\delta_i\cdot \mu(\zeta_p, u, \chi_{i}) = \\
 &\frac{1}{pq}\sum_{i=1}^{e}\delta_i\left( \chi_{i}(u^{pq}) +\Tr_{\QQ(\zeta_q)/\QQ}(\chi_{i}(u^{p})) + \Tr_{\QQ(\zeta_p)/\QQ}(\chi_{i}(u^{q})\zeta_p^{-q}) + \Tr_{\QQ(\zeta_{pq})/\QQ}(\chi_{i}(u)\zeta_p^{-1})\right) = \\
 &\frac{1}{pq}\left(\sum_{i=1}^{e}\delta_i\cdot \chi_{i}(1) +\Tr_{\QQ(\zeta_q)/\QQ}\left(\sum_{i=1}^{e}\delta_i\cdot \chi_{i}(u^{p})\right) - \sum_{i=1}^{e}\delta_i\cdot \chi_{i}(u^{q})  +  \Tr_{\QQ(\zeta_{pq})/\QQ}\left(\sum_{i=1}^{e}\delta_i\cdot \chi_{i}(u)\zeta_p^{-1}\right)\right).\\
\end{align*}

By (a) we have $\nu(1) = \nu(u^p) = 0$. Therefore

\begin{align}\label{first_inequality}
\nu(u^q)& - \Tr_{\QQ(\zeta_{pq})/\QQ}(\nu(u)\zeta_p^{-1}) \nonumber \\ 
=& -(\nu(1) + \Tr_{\QQ(\zeta_{q})/\QQ}(\nu(u^p)) - \nu(u^q) + \Tr_{\QQ(\zeta_{pq})/\QQ}(\nu(u)\zeta_p^{-1})) \nonumber \\
=& -[\delta_{e+1}(\chi_{e+1}(1) + \Tr_{\QQ(\zeta_{q})/\QQ}(\chi_{e+1}(u^p)) - \chi_{e+1}(u^q) + \Tr_{\QQ(\zeta_{pq})/\QQ}(\chi_{e+1}(u)\zeta_p^{-1})) \\
&+ t\sum_{i=1}^e\delta_i(\chi_i(1) + \Tr_{\QQ(\zeta_{q})/\QQ}(\chi_i(u^p)) - \chi_i(u^q) + \Tr_{\QQ(\zeta_{pq})/\QQ}(\chi_i(u)\zeta_p^{-1}))] \nonumber \\
=& pq\left(-\delta_{e+1} \ \mu(\zeta_p,u,\chi_{e+1}) - t\sum_{i=1}^e \delta_i \ \mu(\zeta_p,u,\chi_i) \right) \leq pq. \nonumber
\end{align}
\noindent
Note that this implies that $\nu(u^q)$ is an integer.

%
%

Let $x_1, \ldots ,x_m$ be representatives of conjugacy classes of elements of order $q$ in $G$ and let $y_1, \ldots ,y_r$ be representatives of conjugacy classes of elements of order $p$ in $G$. Since we are assuming that there is no element of order $pq$ in $G$, we have, by Lemma~\ref{lemma_partaugs}, for every ordinary virtual character $\chi$ of $G$
$$\chi(u)=\sum_{j=1}^{r}\varepsilon_{y_j}(u)\chi(y_j) + \sum_{j=1}^{m}\varepsilon_{x_j}(u)\chi(x_j).$$


From now on set $\varepsilon_p=\sum_{j=1}^{r}\varepsilon_{y_j}(u)$. Note that as $\chi_i(y)=\chi_i(y_j)$, for $1\le j\le r$ and every $1 \leq i \leq e+1$, also $\nu(y) = \nu(y_j)$ for every $j$. Hence we obtain using (a)


\begin{align}\label{eq_nuy}
 \nu(u) = \varepsilon_p \cdot\nu(y) + \sum_{j=1}^{m}\varepsilon_{x_j}(u) \nu(x_j)  = \varepsilon_p\cdot\nu(y).
\end{align}

Moreover

\begin{align}\label{eq_nuyq} 
 \nu(u^q)= \nu(y)\sum_{j=1}^{r}\varepsilon_{y_j}(u^q)=\nu(y),
 \end{align}
 
 since $\sum_{j=1}^{r}\varepsilon_{y_j}(u^q)= \varepsilon(u^q) = 1$. Using this information in \eqref{first_inequality} we get 
 $$pq \ge \nu(y)  - \Tr_{\QQ(\zeta_{pq})/\QQ}(\varepsilon_p\cdot\nu(y)\zeta_p^{-1}).$$

As $\nu(y)$ is $p$-rational and $\Tr_{\QQ(\zeta_{pq})/\QQ}(\zeta_p^{-1})=- (q-1)$ the latter becomes
\begin{equation}\label{first_final_inquality}
 pq \ge \nu(y)(1 + \varepsilon_p(q-1)).
\end{equation}

On the other hand, by Theorem~\ref{main_inequality} for $\xi = \zeta_q$ and assuming w.l.o.g $\zeta_p\zeta_q = \zeta_{pq}$ we get $\mu(\zeta_q,u,\chi_1)=0$ and hence 
\begin{align}\label{eq_secondenq}
0\le \delta_{e+1}\cdot \mu(\zeta_{pq}, u, \chi_{e+1}) + t \sum_{i=1}^{e}\delta_i\cdot \mu(\zeta_{pq}, u, \chi_i).
\end{align}
Using Proposition~\ref{LuPa} as above, \eqref{eq_nuyq}, \eqref{eq_nuy} and the fact that $\nu(1)=\nu(u^p)=0$ by (a) we obtain

\begin{align*}
& \delta_{e+1}\cdot \mu(\zeta_{pq}, u, \chi_{e+1}) + t \sum_{i=1}^{e}\delta_i\cdot \mu(\zeta_{pq}, u, \chi_i) =\\ 
  & \frac{1}{pq}\left( \nu(1) + \Tr_{\QQ(\zeta_{q})/\QQ}(\nu(u^p)\zeta_q^{-p}) - \nu(u^q)+ \Tr_{\QQ(\zeta_{pq})/\QQ}(\nu(u)\zeta_{pq}^{-1})\right) =\\ 
&  \frac{1}{pq}\left(  - \nu(u^q)+ \Tr_{\QQ(\zeta_{pq})/\QQ}(\nu(u)\zeta_{pq}^{-1})\right)  =  \frac{1}{pq}\left( - \nu(y) + \varepsilon_p\cdot \nu(y)\Tr_{\QQ(\zeta_{pq})/\QQ}(\zeta_{pq}^{-1})\right) =\\
&\frac{1}{pq}\left(\nu(y)(\varepsilon_p - 1)\right).\\
\end{align*}

Recall that $\nu(u^q) = \nu(y)$ is an integer. So this calculation implies by \eqref{eq_secondenq} that
\begin{equation}\label{second_final_inquality}
 0 \le \nu(y)(\varepsilon_p - 1).
  \end{equation}
%

Now we make use of \eqref{first_final_inquality} and \eqref{second_final_inquality} to achieve our purpose. By Lemma~\ref{congrueneces_pclasses} we have $\varepsilon_p \equiv 0  \mod p$ and $\varepsilon_p \equiv 1 \mod q$. Hence either $\varepsilon_p\ge p$ or $\varepsilon_p\le -p$. Assume first $\varepsilon_p\ge p$. By \eqref{second_final_inquality} we have $\nu(y)\ge 0$ and by \eqref{first_final_inquality}
 $$\nu(y)\le \frac{pq}{1+\varepsilon_p(q-1)}\le \frac{pq}{1+ p(q-1)}= 1 + \frac{p-1}{1+p(q-1)}< 2< p.$$
  
Hence $0\le \nu(y) < p$.

Now assume that $\varepsilon_p\le -p$. By \eqref{second_final_inquality} we obtain $\nu(y)\le 0$. By \eqref{first_final_inquality} moreover $-\nu(y)(1 + \varepsilon_p(q-1)) \ge - pq$. So 
$$\nu(y)\ge \frac{-pq}{-1 -\varepsilon_p(q-1)}\ge \frac{-pq}{-1 +p(q-1)}> -p,$$
where in the last inequality we use that $p \neq 3$.
Hence $-p < \nu(y) \le 0$.

Note that if $\chi$ is a $p$-rational character, then $\chi(y) \equiv \chi(1) \mod p$. This can be seen considering the eigenvalues of $D(y)$ for a representation $D$ realizing $\chi$. As each primitive $p$-th root of unity appears the same number of times as an eigenvalue, it follows that $\chi(1) = \mu(1,y,\chi) + (p-1)\mu(\zeta_p, y, \chi)$ and $\chi(y) = \mu(1,y,\chi) - \mu(\zeta_p,y,\chi)$.
Hence $\chi_i(y)\equiv \chi_i(1) \mod p$. So $\nu(y)\equiv\nu(1)=0 \mod p$ and due to $-p < \nu(y) < p$ we conclude that $\nu(y)=0$, as desired.\\

(c) Finally let $h \in G$ be an arbitrary $p$-singular element. Then $h$ has order $pm$ for some positive integer $m$ not divisible by $p$. Write $h=h_p h_{p'}$, where $h_p$ and $h_{p'}$ denote the $p$-part and the $p'$-part of $h$, respectively. Then $h_p$ is conjugate in $G$ to $y^i$ for some $i$ and $C_G(h_p)=\langle h_p \rangle \times Q$, for a $p'$-subgroup $Q$ of $G$, as $\langle h_p \rangle$ is a central Sylow subgroup in $C_G(h_p)$. Hence $h_{p'}$ lies in $Q=O_{p'}(C_G(h_p))$. By \cite[(7.7) Theorem]{Navarro98}, $\chi_i(h) = \chi_i(h_p)$. Moreover $\chi_i(y) = \chi_i(h_p)$ and so by (b) we obtain

\begin{align*}
 \nu(h) &= \delta_{e+1}\cdot\chi_{e+1}(h) + t\sum_{i=1}^{e} \delta_i\cdot\chi_i(h) \\ 
 &=  \delta_{e+1}\cdot\chi_{e+1}(h_p) + t\sum_{i=1}^{e} \delta_i\cdot\chi_i(h_p) \\ 
 &= \delta_{e+1}\cdot\chi_{e+1}(y) + t\sum_{i=1}^{e} \delta_i\cdot\chi_i(y) = \nu(y) = 0
\end{align*}

Therefore, we conclude that $\nu(g)=0$ for all $g\in G$.
\end{proof}

\section{A number theoretical result on squarefree values of integer polynomials}\label{sec_numbertheory}

This section gives a proof of the number theoretical result which will be useful to us to give a positive answer to the Prime Graph Question for several infinite series' of groups. The proof of the result presented here was sent to us by Roger Heath-Brown which we reproduce in our own words. By a squarefree integer $n$ we mean a number not divisible by the square of any prime or in other words $\mu(n) \neq 0$ where $\mu$ denotes the M\"obius function.

The connection with number theory is explained by the following facts. If $G$ is a finite group of Lie type over a field with $q$ elements, then the order of $G$ is the product of a power of $q$ with the evaluation of certain cyclotomic polynomials $\Phi_i(X) \in \mathbb{Z}[X]$ at $q$, cf. e.g. \cite[Table 24.1]{MalleTesterman}. The group $G$ typically contains elements of order $\Phi_i(q)$, but not always of order $\Phi_i(q)\Phi_j(q)$ where both $\Phi_i(q)$ and $\Phi_j(q)$ are factors in the order formula of $G$. The order of the corresponding simple group is given by dividing $|G|$ by a certain small number $d$. So the application of Theorem~\ref{main_theorem} to a series of finite almost simple groups of Lie-type needs to answer the question if certain $\Phi_i(q)$ are squarefree numbers, at least after ignoring small prime divisors.

The questions of the density of numbers between all integers or all primes for which a given polynomial $f(X) \in \mathbb{Z}[X]$ admits squarefree values is a classical  problem of number theory which was first studied in \cite{Estermann}. Later contributions to the question were given e.g. in \cite{Ricci33, Erdos53, Hooley67, HeathBrown13, Helfgott14, Reuss15} and strong results have been achieved assuming the abc-conjecture or some variations of it \cite{Granville98, Pasten15}. In our applications we will not consider densities, but only whether a number is known to be infinite. Observe that if $f(X)$ is divisible by the square of an irreducible polynomial in $\mathbb{Z}[X]$, then values of $f$ are obviously not squarefree. 

Results in the literature are mostly of two types: Firstly, it is proved that when $f$ is irreducible with some properties, then there are infinitely many primes $p$ such that $f(p)$ is squarefree. Secondly, if $f$ is not divisible by the square of a polynomial and has some properties, then there are infinitely many integers $n$ such that $f(n)$ is squarefree. For our application we need a mix of both types of statements and hence we provide a full proof. In both types of statements the irreducible factors of $f$ are assumed to be of degree at most $3$ and, to our knowledge, there is no irreducible polynomial of degree at least 4 for which it is known that it has infinitely many squarefree values. This also explains our choice of factors: As explained above the irreducible polynomials of interest to us are cyclotomic polynomials $\Phi_i(X)$ and these are of degree 3 or less if and only if $i\in \{1,2,3,4,6\}$. 

To state our result we need some notation. Throughout this section $p$ and $q$ always denote primes. Define a polynomial in $\mathbb{Z}[X]$ as
$$F(X) = (X^2+1)(X^6-1) = \Phi_1(X)\Phi_2(X)\Phi_3(X)\Phi_4(X)\Phi_6(X) $$
where $\Phi_k(X)$ denotes the $k$-th cyclotomic polynomial. For a positive integer $x$ set
$$N(x) = \#\{p \leq x : F(p) \text{ has no divisor } q^2 \text{ such that } q >3 \} $$
and use the standard notation for the logarithmic integral
$$\Li(x) = \int_2^x \frac{dt}{\log t} .$$

Moreover we make standard use of Landau's $O$ notation, i.e. if $M \subset \mathbb{C}$ and $f:M \rightarrow \mathbb{C}$, $g: M \rightarrow \mathbb{R}_{\geq 0}$ are functions, then $f = O(g)$ means that there exists a real constant $C$ such that $|f(z)| \leq Cg(z)$ for all $z \in M$. If $h: M \rightarrow \mathbb{C}$ is another function we write $f = h +O(g)$, if $f-h = O(g)$.

The aim of this section is the proof of the following theorem.
\begin{theorem}\label{theorem:nt}
There is a global positive constant $c$ such that
$$N(x) = c \cdot \Li(x) + O\left(\frac{x}{\log x \cdot (\log \log x)^{\frac{1}{2}}}\right).$$
In particular $\lim\limits_{x \rightarrow \infty} N(x) = \infty$.
\end{theorem}

We first introduce some more notation and prove some auxiliary results. We denote by $\varphi$ Euler's totient function and by $\mu$ the M\"obius function. We fix some integer $x$ bigger than $1$. For positive integers $a$, $d$ and $m$ set
\begin{align*}
\rho(d) &= \# \{a \in \mathbb{Z}/d^2\mathbb{Z} : F(a) \equiv 0 \bmod d^2 \},\\
\pi(x;m,a) &= \# \{p \leq x: p \equiv a \bmod m\}, \\
\xi_1 &= (\log \log x)^\frac{1}{2}.
\end{align*}

We will use the following theorems.

\begin{theorem}\label{th:Siegel-Walfisz} \cite[Corollary 11.21]{MontgomeryVaughan} (Special form of Siegel-Walfisz Theorem)\\
Assume $m \leq (\log x)^7$ and $\gcd(a,m) = 1$. Then
$$\pi(x;m,a) = \frac{\Li(x)}{\varphi(m)} + O\left(\frac{x}{(\log x)^7}\right)$$
\end{theorem}

\begin{theorem}\label{theo:HB97}\cite[Theorem 3]{HeathBrown97}
Let $T(x_1,x_2,x_3)$ be a non-singular integral ternary quadratic form with coefficients bounded in modulus by a number $n$. Assume that also $T(0,x_2,x_3)$ is non-singular. Let $k$ and $R$ be integers and $\varepsilon > 0$. Then the equation $T(x_1,x_2,x_3) = 0$ has only $O((nR)^\varepsilon)$ primitive integer solutions in the cube $\{(x_1,x_2,x_3): |x_i| \leq R, \ i \in \{1,2,3\} \}$ with $x_1 = k$. Here the constant of the error is allowed to depend on $\varepsilon$.
\end{theorem}
Here one calls a solution $T(x_1,x_2,x_3) = 0$ primitive, if $\gcd(x_1, x_2, x_3) = 1$.

\begin{lemma}\label{lemma:xi1xi1log} Let $x$ be a real number.
\begin{enumerate}
\item \label{lem:5.4.1} For $x > e$ we have $\xi_1^{\xi_1} \leq \log x$.
\item \label{lem:5.4.3} $\sum_{m \geq x} \frac{1}{m^2} = O(\frac{1}{x})$ where the sum runs over natural numbers.
\end{enumerate}
\end{lemma}
\begin{proof}
Substituting $y = \sqrt{\log \log x}^{\sqrt{\log \log x}}$ the inequality $\xi_1^{\xi_1} \leq \log x$ becomes $y^y \leq e^{y^2}$ which is true for $y > 0$. Hence the original inequality is true for $x > e$.


For the second claim we can assume $x$ to be an integer. Then
$$\sum_{m \geq x}^N \frac{1}{m^2} \leq \sum_{m \geq x}^N \frac{1}{m(m-1)} = \sum_{m \geq x}^N \left(\frac{1}{m-1} - \frac{1}{m} \right) = \frac{1}{m-1} - \frac{1}{N} $$
which is bound by $O(\frac{1}{m})$.
\end{proof}

\begin{lemma}\label{lemma:c} Let $a$ and $d$ be integers.
\begin{enumerate}
\item\label{lemma:5.5.1} If $F(a) \equiv 0 \bmod d$, then $\gcd(a,d) =1$.
\item\label{lemma:5.5.2} $\rho(q) \leq 8$.
\item\label{lemma:5.5.3} $$c := \prod_{q > 3} \left(1 - \frac{\rho(q)}{\varphi(q^2)} \right). $$
is a positive real number. 
\item\label{lemma:5.5.4}  $$\prod_{q > \xi_1} \left(1 - \frac{\rho(q)}{\varphi(q^2)} \right) = 1 + O(\xi_1^{-1}).$$
\end{enumerate}
\end{lemma}
\begin{proof}
If $q$ is a prime dividing $\gcd(a,d)$, then $F(a) \equiv (a^2+1)(a^6-1) \equiv -1 \bmod q$. Hence also $F(a) \not \equiv 0 \bmod d$. 

To prove $\rho(q) \leq 8$ notice first that the polynomial $F(X)$ has at  most $8$ solutions in the field $\mathbb{Z}/q\mathbb{Z}$. If $q \leq 3$ the statement is clear. If $q > 3$ the congruences $a^2+1 \equiv 0 \bmod q$ and $a^6-1 \equiv 0 \bmod q$ can not hold simultaneously and an easy direct calculation shows that $F(a) \equiv 0 \bmod q^2$ implies that $F(a+bq) \not \equiv 0 \bmod q^2$ for any $1\leq b \leq q-1$.

Notice next that $\varphi(q^2) \geq \frac{q^2}{2}$ which directly follows from $\varphi(q^2) = q(q-1)$. So by the previous claim $c$ is positive, if so is $\prod_{q \geq 5} \left(1 - \frac{16}{q^2} \right)$. This product is positive as the sum $\sum_{q = 5}^\infty \frac{16}{q^2}$ converges.

For the last claim we get, using in particular Lemma~\ref{lemma:xi1xi1log}\eqref{lem:5.4.3}
$$\prod_{q > \xi_1} \left(1 - \frac{\rho(q)}{\varphi(q^2)} \right) \leq \prod_{q > \xi_1} \left(1 - \frac{1}{q^2} \right) = 1 + O\left(\sum_{q > \xi_1} \frac{1}{q^2} \right) = 1 + O(\xi_1^{-1}). $$
\end{proof}

\textit{Proof of Theorem~\ref{theorem:nt}:} Assume $x$ is big enough. We will first establish an estimation of the following number which is clearly an upper bound for $N(x)$: 
\begin{equation}\label{eq:nt1}
f(x) = \# \{ p \leq x : \nexists q \in (3,\xi_1], q^2 \mid F(p)\} = \sum_{d \mid Q}\mu(d) \#\{ p \leq x: d^2 \mid F(p) \}
\end{equation}
where $Q = \prod_{3 < q \leq \xi_1} q.$ The equality in the previous formula follows from the inclusion-exclusion principle.
Note here that $d = 1$ contributes $\#\{p \leq x \}$ in the last sum.

We next study the summands on the right hand side of \eqref{eq:nt1} separately. Note that if $p$ is a prime such that $F(p) \equiv 0 \bmod d^2$ and $p \equiv q \bmod d^2$, then also $F(q) \equiv 0 \bmod d^2$. So contributions to $\#  \{p\leq x : d^2 \mid F(p)\}$ can be viewed as expressions of the form $\pi(x;d^2,a)$ where $a$ is an element of $\rho(d)$. There are, by definition, $\rho(d)$ classes in $\mathbb{Z}/d^2\mathbb{Z}$ contributing. Observe that $Q \leq \xi_1^{\xi_1} \leq \log x$ by Lemma~\ref{lemma:xi1xi1log}\eqref{lem:5.4.1}. Hence applying Siegel-Walsfisz, i.e. Theorem~\ref{th:Siegel-Walfisz}, $\rho(d)$ times we have
\begin{equation}\label{eq:nt2}
\#\{p\leq x : d^2 \mid F(p)\} = \frac{\rho(d)}{\varphi(d^2)}\Li(x) + O\left(\rho(d)\frac{x}{(\log x)^7}\right)
\end{equation}
We next study the contributions of the main and the error term in the last expression in \eqref{eq:nt2}. Now $\rho(d) \leq d^2 \leq Q^2 \leq (\log x)^2$ and hence
\begin{equation}\label{eq:nt3}
O\left(\sum_{d \mid Q} \rho(d)\frac{x}{(\log x)^7}\right) = O\left(\sum_{d \mid Q} \frac{x}{(\log x)^5}\right) = O\left(\frac{Qx}{(\log x)^5}\right) = O\left( \frac{x}{(\log x)^4}\right).
\end{equation}
Moreover, using the definition of $\mu$ and the multiplicativity of $\frac{\rho(d)}{\varphi(d^2)}$ we get for the main term
\begin{equation}\label{eq:nt4}
\sum_{d \mid Q} \mu(d) \frac{\rho(d)}{\varphi(d^2)} \Li(x) = \prod_{3 < q \leq \xi_1}\left(1-\frac{\rho(q)}{\varphi(q^2)} \right)\Li(x)
\end{equation}
So using \eqref{eq:nt2} - \eqref{eq:nt4} we have 
\begin{equation}\label{eq:nt5}
\# \{p \leq x : \nexists q \in (3,\xi_1], q^2 \mid F(p)\} = \prod_{3 < q \leq \xi_1}\left(1-\frac{\rho(q)}{\varphi(q^2)} \right)\Li(x) + O\left( \frac{x}{(\log x)^4}\right) 
\end{equation}
By Lemma~\ref{lemma:c}\eqref{lemma:5.5.4} there is a constant $M$ such that
\begin{equation*}
 \prod_{3 < q \leq \xi_1}\left(1-\frac{\rho(q)}{\varphi(q^2)} \right)\Li(x) = \frac{c}{\prod_{q > \xi_1}\left(1-\frac{\rho(q)}{\varphi(q^2)}\right)} \leq \frac{c}{1+M\xi_1^{-1}}\Li(x).
\end{equation*}
As $\Li(x) = O\left(\frac{x}{\log x}\right)$ follows from the definition of $\Li(x)$ this gives using the definition of $\xi_1$ and the inequality $\frac{1}{1+M\xi_1^{-1}} \leq 1+ M\xi_1^{-1}$ that 
\begin{equation*}
\prod_{3 < q \leq \xi_1}\left(1-\frac{\rho(q)}{\varphi(q^2)} \right)\Li(x) = c\Li(x) + O\left(\frac{x}{\log x (\log \log x)^{\frac{1}{2}}}  \right).
\end{equation*}

Substituting this in \eqref{eq:nt5} we hence have 
\begin{equation}\label{eq:ntupperbound}
\# \{p \leq x : \nexists q \in (3,\xi_1], q^2 \mid F(p)\} = c\Li(x) + O\left(\frac{x}{\log x (\log \log x)^{\frac{1}{2}}}\right)
\end{equation}
which is the desired estimation for $f(x)$. To finish the proof it is enough to show that the number of primes $p \leq x$ for which $F(p)$ is divisible by $q^2$ for $q > \xi_1$ is $O\left(\frac{x}{\log x (\log \log x)^{\frac{1}{2}}} \right)$.

So we now also consider primes $p$ such that $F(p)$ is divisible by $q^2$ for $q > \xi_1$. We will study three intervals of primes separately. Set $\xi_2 = (\log x)^3$, $\xi_3 = x^{\frac{3}{4}}$ and $\xi_4 = 2x$. Note that as $q >3$, if $q^2 \mid F(p)$, then $q^2$ divides at most one of the cyclotomic polynomials $\Phi_i(p)$ which multiply to $F(p)$. Hence $q^2 \leq \Phi_3(p) = p^2 + p + 1$ and so $q \leq 2x$.

We first estimate the contribution of primes in the interval $\xi_1 < q \leq \xi_2$. Then $q^2 \leq (\log x)^6$ and so using Lemma~\ref{lemma:c}\eqref{lemma:5.5.1} and Lemma~\ref{lemma:c}\eqref{lemma:5.5.2} we can apply the Siegel-Walfisz Theorem as before in \eqref{eq:nt2}. We will also again use $\Li(x) = O(\frac{x}{\log x})$ to obtain
\begin{eqnarray*}
\# \{p \leq x: q^2 \mid F(p)\} = \frac{8}{\varphi(q^2)}\Li(x) + O\left(x(\log x)^{-7}\right) = O\left(\frac{q^{-2}x}{\log x}\right).
\end{eqnarray*}
This gives using Lemma~\ref{lemma:xi1xi1log}\eqref{lem:5.4.3} and the definition of $\xi_1$
\begin{equation}\label{eq:nt6}
\sum_{\xi_1 < q \leq \xi_2} \# \{p \leq x: q^2 \mid F(p)\} = \sum_{\xi_1 < q \leq \xi_2} O\left(\frac{q^{-2}x}{\log x}\right) = O\left(\frac{\xi_1^{-1} x}{\log x}\right) = O\left(\frac{x}{\log x (\log \log x)^{\frac{1}{2}}} \right) 
\end{equation}
which is exactly the error term in \eqref{eq:ntupperbound}.

The next interval we consider is $\xi_2 < q \leq \xi_3$. For this, note first that
$$\pi(x;m,a) \leq \# \{n \leq x: n \equiv a \bmod m \} \leq \frac{x}{m} + 1. $$
Hence by Lemma~\ref{lemma:c}\eqref{lemma:5.5.2} and Lemma~\ref{lemma:xi1xi1log}\eqref{lem:5.4.3} there is a constant $M$ such that
\begin{align}\label{eq:nt7}
\sum_{\xi_2 < q \leq \xi_3} \# \{p \leq x: q^2 \mid F(p) \} &\leq \sum_{\xi_2 < q \leq \xi_3} \rho(q)\left(\frac{x}{q^2} + 1\right) \leq \sum_{\xi_2 < q \leq \xi_3} 8\left(\frac{x}{q^2} + 1 \right) \\
 &\leq 8\left(\frac{Mx}{\xi_2} + \xi_3 \right) =  O\left(\frac{x}{(\log x)^3} \right). \nonumber
\end{align}
This is a smaller contribution than the error term in \eqref{eq:ntupperbound}.

We next count pairs of primes $(p,q)$ such that $q^2 \mid F(p)$, $q > \xi_3 = x^\frac{3}{4}$ and $p \leq x$. It follows that $q^2$ does not divide $\Phi_1(p) = p-1$ or $\Phi_2(p) = p+1$. Hence it divides $\Phi_3(p) = p^2+p+1$, $\Phi_4(p) = p^2+1$ or $\Phi_6(p)=p^2-p+1$. Moreover $\Phi_i(p) = q^2n$, for $i \in \{3,4,6 \}$, implies
\begin{equation}\label{eq:nt8}
n \leq \frac{p^2+p+1}{q^2} \leq \frac{x^2+x+1}{\xi_3^2} \leq \frac{2x^2}{x^\frac{3}{2}} = 2x^\frac{1}{2}.  
\end{equation}
We will now apply Theorem~\ref{theo:HB97} with $\varepsilon = \frac{1}{4}$. If $p^2 +1 = nq^2$, then $(1,p,q)$ is a solution of the ternary quadratic form $T_n(x_1, x_2, x_3) = x_1^2 + x_2^2 - nx_3^2$. The coefficients of $T_n$ are bounded in modulus by $n$. So by Theorem~\ref{theo:HB97} the number of solutions of $T_n = 0$ of shape $(1,p,q)$ with $p\leq x$ and $q \leq \xi_4 = 2x$ is bounded by $O((n \cdot 2x)^\frac{1}{4}) = O(x^\frac{3}{8})$. Here we used \eqref{eq:nt8} in the last step. Also by \eqref{eq:nt8} we only have to consider $n \leq 2x^\frac{1}{2}$, so summing over all such $n$ we obtain a contribution of $O(x^\frac{7}{8})$ for pairs of primes $(p,q)$ such that $p^2+1 = q$ and $q > \xi_3$. This error term is smaller than the one in \eqref{eq:ntupperbound}. 

The cases $p^2 - p + 1 = q^2n$ and $p^2+p+1 = q^2n$ can be dealt with in the same way using the forms $T_n(x_1,x_2,x_3) = x_1^2+x_2^2-x_2 - nx_3^2$ and $T_n(x_1,x_2,x_3) = x_1^2 + x_2^2 + x_2 -nx_3^2$, respectively. Overall we conclude that
$$N(x) = c\Li(x) + O\left(\frac{x}{\log x (\log \log x)^\frac{1}{2}}\right).$$
$  \hfill \qed $

\section{Applications}\label{Section_Applications}
We give some applications of our results to the study of the Prime Graph Question for almost simple groups. We note that Theorem~\ref{main_theorem} can be directly applied to obtain the Prime Graph Question for almost simple groups with alternating socle, thus reproducing a result from \cite{BachleMargolisSymmetric}, except in case the socle is $A_6$.

\subsection{Sporadic groups}

We first show how Theorem~\ref{main_theorem} can be used to handle all except two sporadic simple socles, namely the O'Nan and Monster group. We note that for two of the groups for which we prove (PQ) an additional argument using the HeLP-method is required, namely for the Thompson and Held groups.

\begin{proof}[Proof of Corollary~\ref{sporadic}]
Checking the orders of sporadic simple groups and their automorphism groups and the orders of elements in these groups, using e.g. the ATLAS \cite{Atlas} or the Character Table Library \cite{CTblLib} from GAP \cite{GAP}, one directly gets a positive answer for the Prime Graph Question by Theorem~\ref{main_theorem} in case $S$ is: one of the five Mathieu simple groups, the Higman-Sims simple group, one of the four Janko simple groups, one of the three Conway simple groups, the McLaughlin group, the Suzuki group, the Harada-Norton group, one of the three Fischer simple groups, the Rudvalis group or the Lyons simple groups. The Prime Graph Question also follows for the Tits group which is sometimes included in the list of sporadic simple groups. 

\setlength{\parindent}{12pt}

It hence remains to handle the sporadic simple Thompson and Held groups.
For the Thompson group it only remains to prove that the unit group of the integral group ring does not contain elements of order $35$. We will apply a standard argument using character theory to do so, known as the HeLP-method. So let $G$ be the sporadic simple Thompson group and assume $u \in \V(\ZZ G)$ is of order $35$. There are only two non-trivial conjugacy classes of elements of order dividing $35$ in $G$ of order $5$ and $7$ which we call $5a$ and $7a$. If $x \in G$ is not an element in $5a$ or $7a$ then $\varepsilon_x(u) = 0$. We hence have $\varepsilon_{5a}(u) + \varepsilon_{7a}(u) = 1$. We will use a character $\chi$ of an irreducible $248$-dimensional complex representation of $G$ which has value $-2$ on $5a$ and $3$ on $7a$. Note that as $G$ contains exactly one conjugacy class of elements of order $5$ and $7$ we have $\chi(u^7) = \chi(5a)$ and $\chi(u^5) = \chi(7a)$ by Lemma~\ref{lemma_partaugs}. 

We compute two multiplicities of eigenvalues of $u$ under a representation realizing $\chi$ using Lemma~\ref{LuPa}. We will also use $\chi(u) = \varepsilon_{5a}(u)\chi(5a) + \varepsilon_{7a}(u)\chi(7a) = -2\varepsilon_{5}(u) + 3\varepsilon_{7a}(u)$ and $\varepsilon_{7a}(u) = 1 - \varepsilon_{5a}(u)$, which follow from Lemma~\ref{lemma_partaugs} and the fact that $u$ has augmentation $1$.
\begin{align*}
\mu(1, u, \chi) &= \frac{1}{35}\left(\chi(1) + \Tr_{\QQ(\zeta_5)/\QQ}(\chi(u^7)) + \Tr_{\QQ(\zeta_7)/\QQ}(\chi(u^5)) + \Tr_{\QQ(\zeta_{35})/\QQ}(\chi(u))   \right) \\ 
&= \frac{1}{35} \left( 248 - 8 + 18  + 24 (-2\varepsilon_{5a}(u) + 3\varepsilon_{7a}(u)) \right) = \frac{1}{35} (330 - 120\varepsilon_{5a}(u)).
\end{align*}
As this is a non-negative integer we obtain that so is $(330-120\varepsilon_{5a}(u))$ implying $\varepsilon_{5a}(u) \leq 2$.
Next we obtain
\begin{align*}
\mu(\zeta_5, u, \chi) &= \frac{1}{35}\left(\chi(1) + \Tr_{\QQ(\zeta_5)/\QQ}(\chi(u^7)\zeta_5^{-7}) + \Tr_{\QQ(\zeta_7)/\QQ}(\chi(u^5)) + \Tr_{\QQ(\zeta_{35})/\QQ}(\chi(u)\zeta_5^{-7})   \right) \\ 
&= \frac{1}{35} \left( 248 + 2 + 18  - 6 (-2\varepsilon_{5a}(u) + 3\varepsilon_{7a}(u)) \right) = \frac{1}{35} (250 + 30\varepsilon_{5a}(u)).
\end{align*}
So $(250 + 30\varepsilon_{5a}(u))$ is a non-negative integer, implying $\varepsilon_{5a}(u) \geq -8$. But the bounds we obtained for $\varepsilon_{5a}(u)$ are not compatible with the conditions $\varepsilon_{5a}(u) \equiv 0 \mod 5$ and $\varepsilon_{5a}(u) \equiv 1 \mod 7$ which we obtain from Lemma~\ref{congrueneces_pclasses}.

For the sporadic Held group it remains to show that the normalized unit group of the group ring of the automorphism group does not contain elements of order $35$. Using the HeLP-method this has been done in \cite{KimmerleKonovalov15} and it can also be checked using the GAP-package \cite{BachleMargolisHeLP}.
\end{proof}

\begin{remark}
To answer the Prime Graph Question for the Monster group we are left after the application of Theorem~\ref{main_theorem} with the following orders of normalized torsion units we need to exclude: $5 \cdot 13$, $7 \cdot 11$, $7 \cdot 13$ and $11 \cdot 13$.

For the O'Nan simple group and its automorphism group it would remain to exclude the existence of normalized torsion units of order $3 \cdot 7$. In \cite{BovdiKonovalovONan} it is claimed that this can be achieved using the HeLP-method, but we were unable to confirm this. The inequalities the authors derive in this article for units of order $21$ are:
\begin{itemize}
\item $\frac{1}{21}(98493 + 312\varepsilon_3) \in \mathbb{Z}_{\geq 0},$
\item $\frac{1}{21}(98415 + 26\varepsilon_3) \in \mathbb{Z}_{\geq 0},$
\item $\frac{1}{21}(98415 - 156\varepsilon_3) \in \mathbb{Z}_{\geq 0},$
\end{itemize}
where $\varepsilon_3$ denotes the partial augmentation of a potential unit $u$ of order $21$ at the conjugacy class of elements of order $3$ in the O'Nan simple group. Denote by $\varepsilon_7$ the sum of the partial augmentations of $u$ at elements of order $7$. These inequalities are correct, but clearly have non-trivial integral solutions, such as e.g. $(\varepsilon_3, \varepsilon_7) = (-6, 7)$, which also satisfies other known restrictions on partial augmentations of torsion units.

\setlength{\parindent}{12pt}

It seems worth pointing out that the positive results Bovdi, Konovalov et al. achieved for the sporadic simple groups in regard to the Prime Graph Question using the HeLP-method have been confirmed by Verbeken \cite{Brecht} using an implementation of the HeLP-method in a publicly available GAP-package \cite{BachleMargolisHeLP}. 
\end{remark}

\setlength{\parindent}{12pt}

\subsection{Simple groups of Lie type}

We proceed to give some applications to almost simple groups of Lie type. To simplify notation we introduce the function $\alpha$ mapping a positive integer to its biggest divisor coprime to $6$, i.e.
$$\alpha: \mathbb{Z}_{>0} \rightarrow \mathbb{Z}_{>0}, \ \ \alpha(n) = \max\{d \in \mathbb{Z} \ : \ \gcd(d,6)=1,\ d \mid n \}.$$
Before we will use the number theoretical results from Section~\ref{sec_numbertheory} we study the series' of groups for which they will be applied separately.

\begin{lemma}\label{lemma:PSL4}
Let $G = \PSL(4,q)$ with $q=p^f$ and let $\alpha(f) = c$. Assume $c$ is squarefree and coprime to $(q^2+q+1)(q^2+1)$. If $\alpha((q^2+q+1)(q^2+1))$ is also squarefree, then the Prime Graph Question has a positive answer for any almost simple group with socle $G$. 
\end{lemma}
\begin{proof}
We will use facts about $G$ given in \cite[Section 3.3]{Wilson}. Let $H$ be an almost simple group with socle $G$ and let $r$ and $s$ be primes dividing the order of $H$. We have
$$|G| = \frac{1}{\gcd(4,q-1)} q^6 (q^2-1)(q^3-1)(q^4-1) = \frac{1}{\gcd(4,q-1)} q^6 (q-1)^3(q+1)^2(q^2+q+1)(q^2+1)$$
and $\alpha(|\Out(G)|) = c$. Moreover $G$ has a subgroup mapping onto $\operatorname{PSL}(2,q)\times \operatorname{PSL}(2,q)$. Hence, if $r$ and $s$ divide $q(q-1)(q+1)$, then $H$ contains an element of order $rs$. So if $H$ does not contain an element of order $rs$, then $r$ or $s$ is a divisor of $\alpha((q^2+q+1)(q^2+1)c)$. This number is squarefree by assumption and so the Prime Graph Question has a positive answer for $H$ by Theorem~\ref{main_theorem}.
\end{proof}

\begin{lemma}\label{lemma:PSU4}
Let $G = \operatorname{PSU}(4,q)$ with $q=p^f$ and let $\alpha(f) = c$. Assume $c$ is squarefree and coprime to $(q^2+1)(q^2-q+1)$. If $\alpha((q^2+1)(q^2-q+1))$ is also squarefree, then the Prime Graph Question has a positive answer for any almost simple group with socle $G$. 
\end{lemma}
\begin{proof}
The necessary facts about $G$ are contained in \cite[Section 3.6]{Wilson}. We have
$$|G| = \frac{1}{\gcd(4,q+1)} q^6 (q^2-1)(q^3+1)(q^4-1) = \frac{1}{\gcd(4,q+1)} q^6 (q-1)^2(q+1)^3(q^2+1)(q^2-q+1)$$
and $\alpha(|\Out(G)|) = c$. The proof now follows in the same way as in Lemma~\ref{lemma:PSL4} using the subgroup $\operatorname{PSU}(2,q) \times \operatorname{PSU}(2,q)$.
\end{proof}

We next handle groups having symplectic socle. Recall that
$$|\operatorname{PSp}(2n,q)| =  \frac{1}{\gcd(2,q-1)}q^{n^2}(q^2-1)(q^4-1)...(q^{2n}-1). $$

\begin{lemma}\label{lemma:PSp4}
Let $G = \operatorname{PSp}(4,q)$ with $q=p^f$ and let $\alpha(f) = c$. Assume $c$ is squarefree and coprime to $(q^2+1)$. If $\alpha(q^2+1)$ is also squarefree, then the Prime Graph Question has a positive answer for any almost simple group with socle $G$. 
\end{lemma}
\begin{proof}
We will use facts about $G$ given in \cite[Section 3.5]{Wilson}. 
We have
$$|G| = \frac{1}{\gcd(2,q-1)} q^4 (q-1)^2(q+1)^2(q^2+1) $$
and $\alpha(|\Out(G)|) = c$. Moreover $G$ has a subgroup mapping onto $\operatorname{PSp}(2,q)\times \operatorname{PSp}(2,q)$. We hence can argue as in the proof of Lemma~\ref{lemma:PSL4}.
\end{proof}

\begin{lemma}\label{lemma:PSp6}
Let $G = \operatorname{PSp}(6,q)$ with $q=p^f$ and let $\alpha(f) = c$. Assume $c$ is squarefree and coprime to $(q^2+q+1)(q^2-	q+1)$. If $\alpha((q^2+q+1)(q^2-q+1))$ is also squarefree, then the Prime Graph Question has a positive answer for any almost simple group with socle $G$. 
\end{lemma}
\begin{proof}
We can argue as in the proof of Lemma~\ref{lemma:PSp4} using the fact that $G$ contains a subgroup mapping onto $\operatorname{PSp}(2,q) \times \operatorname{PSp}(4,q)$ and that $\operatorname{PSp}(4,q)$ contains an element of order $\alpha(q^2+1)$.
\end{proof}

We next deal with orthogonal groups.

\begin{lemma}\label{lemma:O7}
Let $G = \operatorname{P\Omega}(7,q)$ with $q=p^f$ and let $\alpha(f) = c$. Assume $c$ is squarefree and coprime to $(q^2+q+1)(q^2-q+1)$. If $\alpha((q^2+q+1)(q^2-q+1))$ is also squarefree, then the Prime Graph Question has a positive answer for any almost simple group with socle $G$. 
\end{lemma}
\begin{proof}
The necessary facts about $G$ are contained in \cite[Section 3.7]{Wilson}. We have
\begin{align*}
|G| &= \frac{1}{\gcd(2,q+1)} q^9 (q^2-1)(q^4-1)(q^6-1) \\
 &= \frac{1}{\gcd(2,q+1)} q^6 (q-1)^3(q+1)^3(q^2+1)(q^2+q+1)(q^2-q+1)
\end{align*}
and $\alpha(|\Out(G)|) = c$. The proof now follows in the same way as in Lemma~\ref{lemma:PSL4} using a subgroup mapping onto $\operatorname{P\Omega}(3,q) \times \operatorname{P\Omega}^-(4,q)$ and the fact that $\operatorname{P\Omega}^-(4,q)$ contains an element of order $\alpha(q^2+1)$.
\end{proof}

\begin{lemma}\label{lemma:O+8}
Let $G = \operatorname{P\Omega}^+(8,q)$ with $q=p^f$ and let $\alpha(f) = c$. Assume $c$ is squarefree and coprime to $(q^2+q+1)(q^2-q+1)$. If $\alpha((q^2+q+1)(q^2-q+1))$ is also squarefree, then the Prime Graph Question has a positive answer for any almost simple group with socle $G$. 
\end{lemma}
\begin{proof}
As before we use \cite[Section 3.7]{Wilson}. We have
\begin{align*}
|G| &= \frac{1}{\gcd(4,q^4-1)} q^{12} (q^2-1)(q^4-1)(q^6-1)(q^4-1) \\
 &= \frac{1}{\gcd(4,q^4-1)} q^{12} (q-1)^4(q+1)^4(q^2+1)^2(q^2+q+1)(q^2-q+1)
\end{align*}
and $\alpha(|\Out(G)|) = c$. The proof now follows in the same way as in Lemma~\ref{lemma:PSL4} using a subgroup mapping onto $\operatorname{P\Omega}(3,q) \times \operatorname{P\Omega}(5,q)$ and the fact that the group $\operatorname{P\Omega}(5,q)$ contains an element of order $\alpha(q^2+1)$.
\end{proof}

Finally we also deal with one series of exceptional groups.

\begin{lemma}\label{lemma:G2}
Let $G = G_2(q)$ with $q=p^f$ and let $\alpha(f) = c$. Assume $c$ is squarefree and coprime to $(q^2+q+1)(q^2-q+1)$. If the $\alpha((q^2+q+1)(q^2-q+1))$ is also squarefree, then the Prime Graph Question has a positive answer for any almost simple group with socle $G$. 
\end{lemma}
\begin{proof}
We will use facts about $G$ given in \cite[Section 4.3]{Wilson}. Let $H$ be an almost simple group with socle $G$ and let $r$ and $s$ be primes dividing the order of $H$. We have
$$|G| = q^6(q^6-1)(q^2-1) = q^6(q-1)^2(q+1)^2(q^2+q+1)(q^2-q+1) $$
and $\Out(G) \cong C_f$, if $p \neq 3,$ or $\Out(G) \cong C_{2f}$, if $p=3$. Moreover $G$ has a quotient with a subgroup isomorphic to $\PSL(2,q)\times \PSL(2,q)$. Hence, if $r$ and $s$ divide $q(q-1)(q+1)$, then $H$ contains an element of order $rs$. So if $H$ does not contain an element of order $rs$, then $r$ or $s$ is a divisor of $\alpha((q^2+q+1)(q^2-q+1)c)$. This number is squarefree by assumption and so the Prime Graph Question has a positive answer for $H$ by Theorem~\ref{main_theorem}.
\end{proof}

The proof of Corollary~\ref{Series} is now a direct application of Theorem~\ref{theorem:nt} in the situations described in the preceding lemmas.

\begin{proof}[Proof of Corollary~\ref{Series}]
By Lemmas~\ref{lemma:PSL4}-\ref{lemma:G2} it is sufficient to show that there are infinitely many primes $p$ such that $(p^2+p+1)(p^2+1)(p^2-p+1)$ is squarefree. This is true by Theorem~\ref{theorem:nt}.
\end{proof}

It seems probable that a better understanding of number theoretical sieves can add further infinite series of almost simple groups of Lie type for which Theorem~\ref{main_theorem} can answer the Prime Graph Question.

\subsection{Groups from the GAP character table library}
The proof of Corollary~\ref{cor:Library} is a direct application of Theorem~\ref{main_theorem}.

We finish by providing an overview of the almost simple groups from the GAP Character Table Library for which the Prime Graph Question has already been studied before and for which of those Theorem~\ref{main_theorem} is sufficient to obtain a positive result. We exclude groups whose socle is a sporadic group, an alternating group or a group of type $\PSL(2,p)$ or $\PSL(2,p^2)$ for $p$ a prime.

\newpage

\begin{center}
\begin{longtable}{|p{.45\textwidth} | p{.2\textwidth}|  p{.25\textwidth}|} \hline
{\bf (PQ)} holds by Theorem~\ref{main_theorem} & {\bf (PQ)} known with other result & {\bf (PQ)} not known  (all studied in \cite{BachleMargolis4primaryI}) \\ \hline
$\PSL(2,16)$ $\PSL(2,32)$, $\PSL(3,3)$, $\PSL(3,5)$ & $\PSL(2,8)$ \cite{KimmerleKonovalov} & $\PSL(2,27)$, $\PSL(2,81)$,   \\
$\PSL(4,3)$  &  $\PSL(3,4)$ \cite{BachleMargolis4primaryII} & $\PSL(3,7)$, $\PSL(3,8)$ \\
$\operatorname{PSp}(4,4)$, $\operatorname{PSp}(4,5)$, $\operatorname{PSp}(4,9)$, $\operatorname{PSp}(6,2)$ & & $\operatorname{PSp}(4,7)$ \\
$\operatorname{PSU}(3,3)$, $\operatorname{PSU}(3,4)$,  $\operatorname{PSU}(3,7)$ & $\operatorname{PSU}(3,5)$ \cite{BachleMargolis4primaryI} & \\
$\operatorname{PSU}(3,8)$, $\operatorname{PSU}(3,9)$, $\operatorname{PSU}(4,2)$, $\operatorname{PSU}(4,3)$ & & \\
$\operatorname{PSU}(4,4)$, $\operatorname{PSU}(4,5)$, $\operatorname{PSU}(5,2)$ & & \\
$\operatorname{P\Omega}^+(8,2)$ & & \\
$G_2(3)$, $G_2(4)$, $Sz(8)$, ${}^3D_4(2)$ & ${}^2F_4(2)'$ \cite{BachleMargolis4primaryI} & $Sz(32)$ \\
\hline 
\caption{Groups from the GAP character table library for which (PQ) has been studied before.}\label{table2}
\end{longtable}
\end{center}

\vspace*{-1cm}

\textbf{Acknowledgments:} We are very thankful to Roger Heath-Brown for his proof of Theorem~\ref{theorem:nt} and Dan Carmon for help with the number theory involved. We also thank Gunter Malle for useful conversations on the structure of groups of Lie type. Moreover, we thank the referee for his valuable comments which helped to improve the paper significantly.

\bibliographystyle{amsalpha}
\bibliography{sporadicos}

\end{document}